\documentclass{amsart}
\usepackage{amsthm,amsmath,amsfonts,mathrsfs,amssymb, eucal,comment, enumerate,tikz,float,graphicx,enumitem}

\restylefloat*{figure}

\numberwithin{equation}{section}
\usetikzlibrary{arrows}

\theoremstyle{plain}
\newtheorem{theorem}{Theorem}[section]
\newtheorem{proposition}[theorem]{Proposition}
\newtheorem{lemma}[theorem]{Lemma}
\newtheorem{corollary}[theorem]{Corollary}

\theoremstyle{definition}
\newtheorem{definition}[theorem]{Definition}
\newtheorem{remark}[theorem]{Remark}
\newtheorem{example}[theorem]{Example}
\newtheorem{notation}[theorem]{Notation}

\newlist{case}{enumerate}{3}
\setlist[case, 1]{label={Case (\arabic*).},leftmargin=*}

\title[Dimensions of automorphism group schemes of $F$-crystals]{Dimensions of automorphism group schemes of finite level truncations of $F$-cyclic $F$-crystals}
\author[Z. Ding]{Zeyu Ding}
\address{Department of Computer Science and Engineering, Pennsylvania State University, University Park, PA 16802}
\email{zyding@cse.psu.edu}
\author[X. Xiao]{Xiao Xiao}
\address{Mathematics Department, Utica College, 1600 Burrstone Road, Utica, NY 13502}
\email{xixiao@utica.edu}

\begin{document}
\begin{abstract}
Let $\mathcal{M}_{\pi}$ be an $F$-cyclic $F$-crystal $\mathcal{M}_{\pi}$ over an algebraically closed field defined by a permutation $\pi$ and a set of prescribed Hodge slopes. We prove combinatorial formulas for the dimension $\gamma_{\mathcal{M}_{\pi}}(m)$ of the automorphism group scheme of $\mathcal{M}_{\pi}$ at finite level $m$ and the number of connected components of the endomorphism group scheme of $\mathcal{M}_{\pi}$ at finite level $m$. As an application, we show that if $\mathcal{M}_{\pi}$ is a nonordinary Dieudonn\'e module defined by a cycle $\pi$, then $\gamma_{\mathcal{M}_{\pi}}(m+1) - \gamma_{\mathcal{M}_{\pi}}(m) < \gamma_{\mathcal{M}_{\pi}}(m) - \gamma_{\mathcal{M}_{\pi}}(m-1)$ for all $1 \leq m \leq n_{\mathcal{M}_{\pi}}$, where $n_{\mathcal{M}_{\pi}}$ is the isomorphism number of $\mathcal{M}_{\pi}$.
\end{abstract}

\maketitle

\section{Introduction}
Fix an integer $m \geq 1$, a prime number $p$ and an algebraically closed field $k$ of characteristic $p$ throughout this paper. Let $D$ be a $p$-divisible group over $k$ of codimension $c$ and dimension $d$. The isomorphism number $n_D$ of $D$ is the smallest nonnegative integer such that for every $p$-divisible group $C$ over $k$ of the same codimension and dimension as $D$, if $C[p^{n_D}]$ is isomorphic to $D[p^{n_D}]$, then $C$ is isomorphic to $D$. For every integer $n \geq 0$, let $\mathbf{Aut}(D[p^n])$ be the smooth affine group scheme over $k$ of automorphisms of $D[p^n]$ and let $\gamma_D(n) = \text{dim}(\mathbf{Aut}(D[p^n]))$. Hence $\gamma_D(0) = 0$. Moreover, $\gamma_D(1) = 0$ if and only if $D$ ordinary (i.e., $D \cong (\mathbb{Q}_p/\mathbb{Z}_p)^c \oplus \mathbf{\mu}_{p^{\infty}}^d$). For every integer $l > 0$, let 
\[\mathcal{S}_D(l) := (\gamma_D(n+l) - \gamma_D(n))_{n \geq 0}\] be an infinite sequence of nonnegative integers. Gabber and Vasiu proved that:

\begin{theorem}{\cite[Theorem 1]{Vasiu:dimensions}} \label{theorem:vasiugabberbasic}
Let $D$ be a nonordinary $p$-divisible group over $k$. For every integer $l > 0$, the sequence $\mathcal{S}_D(l)$ is nonincreasing and we have
\[0 < \gamma_D(1) < \gamma_D(2) < \cdots < \gamma_D(n_D) = \gamma_D(n_D+1) = \cdots.\]
\end{theorem}

For every integer $l>0$, as $\gamma_D(n_D+l) = \gamma_D(n_D+l-1) = \gamma_D(n_D)$ and $\gamma_D(n_D) > \gamma_D(n_D-1)$, we get that
\[0 = \gamma_D(n_D+l) - \gamma_D(n_D) < \gamma_D(n_D+l-1) - \gamma_D(n_D-1),\]
which is the only strictly decreasing part of $\mathcal{S}_D(l)$ guaranteed by Theorem \ref{theorem:vasiugabberbasic}. We want to study whether the finite sequence $\mathcal{S}^*_D(l) := (\gamma_D(n+l) - \gamma_D(n))_{0 \leq n < n_D}$ is strictly decreasing. If $\mathcal{S}^*_D(1)$ is strictly decreasing, then $\mathcal{S}^*_D(l)$ is strictly decreasing for every integer $l \geq 2$ because
\[\gamma_D(n+l) - \gamma_D(n) = \sum_{i=1}^l (\gamma_D(n+i)-\gamma_D(n+i-1)).\]
In this paper, we show that for a certain family of $p$-divisible groups $D_{\pi}$ that are defined by a cycle $\pi$ and a set of prescribed Hodge slopes, the sequence $\mathcal{S}^*_{D_{\pi}}(1)$ is strictly decreasing and hence $\mathcal{S}^*_{D_{\pi}}(l)$ is strictly decreasing for every integer $l >0$.

\subsection{$F$-cyclic $p$-divisible Groups}
Let $W(k)$ be the ring of $p$-typical Witt vectors with coefficients in $k$. Let $W_m(k) = W(k)/(p^m)$ be the ring of truncated $p$-typical Witt vectors of length $m$ with coefficients in $k$. Let $B(k) = W(k)[1/p]$ be the field of fractions of $W(k)$. Let $\sigma$ be the Frobenius automorphism of $k$, $W(k)$, $W_m(k)$ and $B(k)$. 

An $F$-crystal over $k$ is a pair $\mathcal{M} = (M, \varphi)$, where $M$ is a free $W(k)$-module of finite rank $r$ and $\sigma: M \to M$ is a $\sigma$-linear monomorphism. If $pM \subset \varphi(M) \subset M$, then $\mathcal{M}$ is called a (contravariant) Dieudonn\'e module. For every $F$-crystal $\mathcal{M}$, the isomorphism number $n_{\mathcal{M}}$ is the smallest nonnegative integer such that for every $W(k)$-linear automorphism $g$ of $M$, if $g \equiv 1$ modulo $p^{n_{\mathcal{M}}}$, then the $F$-crystal $(M, g\varphi)$ is isomorphic to $\mathcal{M}$; see \cite[Main Theorem A]{Vasiu:CBP} for the existence of $n_{\mathcal{M}}$. Let $\mathbf{Aut}_m(\mathcal{M})$ (resp. $\mathbf{End}_m(\mathcal{M})$) be the smooth affine group scheme over $k$ whose $k$-valued points is the group of automorphisms (resp. group of endomorphisms) of $\mathcal{M}$ modulo $p^m$; see Subsection \ref{subsection:basicsetup} for precise definitions. Let $\gamma_{\mathcal{M}}(m)$ be the dimension of $\mathbf{Aut}_m(\mathcal{M})$ and $\gamma_{\mathcal{M}}(0) = 0$. 

It is well-known that the category of $p$-divisible groups over $k$ is anti-equivalent to the category of Dieudonn\'e modules over $k$.  If $\mathcal{M}$ is the Dieudonn\'e module of some $p$-divisible group $D$, then $n_{\mathcal{M}} = n_{D}$ and $\gamma_{\mathcal{M}}(n) = \gamma_D(n)$ for every integer $n \geq 0$. 

We recall the following definition from \cite[Definition 1.5.1]{Vasiu:reconstructing}.

\begin{definition}
Recall that $c$ and $d$ are nonnegative integers such that $r:=c+d$. Let $B = (v_1, v_2, \dots, v_r)$ be an ordered $W(k)$-basis of $M$ and let $\pi$ be a permutation of the set $I_r := \{1,2,\dots,r\}$. Let $\mathcal{M}_{c, d, B, \pi} = (M, \varphi_{c,d,B, \pi})$ be the Dieudonn\'e module over $k$ with the property that $\varphi_{c,d,B,\pi}(v_i) = v_{\pi(i)}$ if $i \in \{1, \dots, c\}$ and $\varphi_{c,d,B,\pi}(v_i) = pv_{\pi(i)}$ if $i \in \{c+1, \dots, r\}$. A Dieudonn\'e module $\mathcal{M}$ of codimension $c$ and dimension $d$ is said to be \emph{$F$-cyclic} (resp. \emph{$F$-circular}) if there exist a permutation (resp. an $r$-cycle permutation) $\pi$ on $I_r$ and an ordered $W(k)$-basis $B$ of $M$ such that $\mathcal{M}$ is isomorphic to $\mathcal{M}_{c, d, B, \pi}$. A $p$-divisible group $D$ is said to be \emph{$F$-cyclic} (resp. \emph{$F$-circular}) if the Dieudonn\'e module of $D$ is $F$-cyclic (resp. $F$-circular). When $c$, $d$ and $B$ are understood, we let $D_{\pi}$ be the $p$-divisible group of $\mathcal{M}_{\pi} := \mathcal{M}_{c, d, B, \pi}$.
\end{definition}

Kraft's work \cite{Kraft1} on the classification of finite group schemes over $k$ that are annihilated by $p$ implicitly implies that for every $p$-divisible group $D$ over $k$, there exists a permutation $\pi$ such that $D_{\pi}[p] \cong D[p]$. In \cite{Vasiu:modp}, $F$-cyclic $p$-divisible groups are studied using the language of Weyl groups. Vasiu proved that if $\pi_1, \pi_2$ are two permutations on $I_r$, then the $p$-divisible groups $D_{\pi_1}$ and $D_{\pi_2}$ are isomorphic if and only if $D_{\pi_1}[p]$ and $D_{\pi_2}[p]$ are isomorphic; see \cite[1.3 Basic Theorem B (a)]{Vasiu:modp}. 

As an application of the classification of $D$-truncations mod $p$ of the so-called Shimura $F$-crystals over $k$, Vasiu provides a formula for $\gamma_{D_{\pi}}(1)$; see \cite[1.2 Basic Theorem A]{Vasiu:modp}. In this paper, we prove an explicit formula for $\gamma_{\mathcal{M}_{\pi}}(m)$ for all $F$-cyclic $F$-crystals $\mathcal{M}_{\pi}$.

\subsection{Main Results}
Our main results pertain to all $F$-cyclic $F$-crystals (see Definition \ref{definition:fcyclicfcrystal}) but for the sake of simplicity, here we only state their version for $F$-cyclic Dieudonn\'e modules. Let $\mathcal{M}_{\pi}$ be an $F$-cyclic Dieudonn\'e module of rank $r=c+d$, codimension $c$, and dimension $d$. Let $\mathcal{B}_{\pi \times \pi}$ be the set of orbits of $\pi \times \pi$ on $I_r^2$. For every orbit $\mathcal{O} = \{(i_1,j_1),(i_2,j_2), \dots, (i_s,j_s)\}$ in $\mathcal{B}_{\pi \times \pi}$, i.e., $(\pi \times \pi)(i_t,j_t)=(i_{t+1},j_{t+1})$ for $t \in I_{s-1}$ and $(\pi \times \pi)(i_{t+1},j_{t+1})=(i_1,j_1)$, define
\[\epsilon_{\mathcal{O}} = (\epsilon_1, \epsilon_2, \dots, \epsilon_s) := (e_{i_1}-e_{j_1}, e_{i_2}-e_{j_2}, \dots, e_{i_s}-e_{j_s}),\]
where $e_1=e_2 = \cdots = e_c = 0$ and $e_{c+1} = e_{c+2} = \cdots = e_{r} = 1$ are the Hodge slopes of $\mathcal{M}_{\pi}$. Note that $\epsilon_i \in \{0,\pm1\}$ for all $1 \leq i \leq s$. Let $|\mathcal{O}| = s \geq 1$ denote the length of $\mathcal{O}$. For every positive integer $\lambda$, let $a_{\lambda}(\epsilon_{\mathcal{O}})$ be the number of free linear segments of level $\lambda$ in $\epsilon_{\mathcal{O}}$ as introduced in Definition \ref{definition:freelinearsegment}. For every nonnegative integer $\lambda$, let $\mathcal{C}_{\pi}(\lambda) \subset \mathcal{B}_{\pi \times \pi}$ be the set of orbits $\mathcal{O}$ such that $\epsilon_{\mathcal{O}}$ is a circular sequence of level $\lambda$ as introduced in Definition \ref{definition:circularsequencelevel}.

\begin{theorem}[Main Result, Dieudonn\'e Module Case] \label{theorem:mainintro}
Let $\mathcal{M}_{\pi}$ be an $F$-cyclic Dieudonn\'e module over $k$. Using the above notation, for every integer $m \geq 1$, the dimension of $\mathbf{Aut}_m(\mathcal{M}_{\pi})$ is equal to
\[\sum_{\mathcal{O} \in \mathcal{B}_{\pi \times \pi}} \sum_{\lambda=1}^m a_{\lambda}(\epsilon_{\mathcal{O}}),\]
and the number of connected components of $\mathbf{End}_m(\mathcal{M}_{\pi})$ is equal to $p^b$, where
\[b = \sum_{\lambda = 0}^{m-1} \sum_{\mathcal{O} \in \mathcal{C}_{\pi}(\lambda)} (m-\lambda)|\mathcal{O}|.\qedhere\]
\end{theorem}
See Theorem \ref{theorem:main} for the more general $F$-crystal case, which will imply the Dieudonn\'e module case stated above. We state one application of the main result.


\begin{corollary}[Theorem \ref{theorem:strictinequalitycircular}]
If $\mathcal{M}$ is a nonordinary $F$-circular Dieudonn\'e module over $k$ of rank $r \geq 2$ (and thus $n_{\mathcal{M}} \geq 1$), then the sequence $\mathcal{S}^*_{\mathcal{M}}(1)$ is strictly decreasing, i.e., we have
\[\gamma_{\mathcal{M}}(1) - \gamma_{\mathcal{M}}(0) > \gamma_{\mathcal{M}}(2) - \gamma_{\mathcal{M}}(1) > \cdots > \gamma_{\mathcal{M}}(n_{\mathcal{M}}) - \gamma_{\mathcal{M}}(n_{\mathcal{M}}-1)>0.\]
\end{corollary}

\begin{remark}
Although Theorem \ref{theorem:vasiugabberbasic} can be generalized to the $F$-crystal case (see \cite[Proposition 2.11, Theorem 3.15]{Xiao:vasiuconjecture}) and Theorem \ref{theorem:mainintro} applies to $F$-cyclic $F$-crystal as well, it is not true that the sequence $\mathcal{S}^*_{\mathcal{M}}(l)$ is always strictly decreasing if $\mathcal{M}$ is an $F$-circular $F$-crystal. We construct an $F$-circular $F$-crystal $\mathcal{M}$ of rank $2$ in Example \ref{example:counterexamplefcrystal} such that $\mathcal{S}^*_{\mathcal{M}}(1)$ is a constant sequence of arbitrary finite length.
\end{remark}

\begin{notation}
All $p$-divisible groups, Dieudonn\'e modules, and $F$-crystals in this paper are over $k$.
\end{notation}

\section{Endomorphism Group Schemes at Finite Level}

\subsection{Basic Setup} \label{subsection:basicsetup}
Let $\mathcal{M} = (M, \varphi)$ be an $F$-crystal of rank $r>0$. For each integer $m \geq 1$, there is a smooth affine group scheme $\mathbf{Aut}_m(\mathcal{M})$ over $k$ such that its group of $k$-valued points is the group of automorphisms of $\mathcal{M}$ modulo $p^m$. More precisely, we have
\begin{equation*}
\begin{gathered}
\mathbf{Aut}_m(\mathcal{M})(k) = \{\bar{f} \in \mathrm{GL}_{M/p^mM}(W_m(k)) \; | \; \exists \; f \in \mathrm{GL}_M(W(k)) \\ 
\textrm{ such that } P(f) = \bar{f}  \textrm{ and } \varphi f \varphi^{-1} \in  f +p^m \mathrm{End}_M(W(k)\},
\end{gathered}
\end{equation*}
where $P: \mathrm{GL}_M(W(k)) \to \mathrm{GL}_{M/p^mM}(W_m(k))$ is the restriction modulo $p^m$ epimorphism. Similarly, there exists a smooth affine group scheme $\mathbf{End}_m(\mathcal{M})$ over $k$ such that its group of $k$-valued points is the additive group of endomorphisms of $\mathcal{M}$ modulo $p^m$. More precisely, we have
\begin{equation*}
\begin{gathered}
\mathbf{End}_m(\mathcal{M})(k) = \{\bar{f} \in \mathrm{End}_{M/p^mM}(W_m(k)) \; | \; \exists \; f \in \mathrm{End}_M(W(k))\\
\textrm{ such that } P'(f) = \bar{f} \textrm{ and } \varphi f \varphi^{-1} \in f + p^m \mathrm{End}_M(W(k)) \},
\end{gathered}
\end{equation*}
where $P': \mathrm{End}_M(W(k)) \to \mathrm{End}_{M/p^mM}(W_m(k))$ is the restriction modulo $p^m$ epimorphism. We refer to \cite[Section 2.4]{Xiao:vasiuconjecture} for the detail construction of $\mathbf{Aut}_m(\mathcal{M})$ and $\mathbf{End}_m(\mathcal{M})$. Note that $\mathbf{Aut}_m(\mathcal{M})$ is an open subscheme of $\mathbf{End}_m(\mathcal{M})$, thus the dimension of $\mathbf{Aut}_m(\mathcal{M})$ and the dimension of $\mathbf{End}_m(\mathcal{M})$ are equal.

\begin{definition} \label{definition:fcyclicfcrystal}
Let $\mathcal{M}$ be an $F$-crystal, and $\mathcal{E} = (e_1, e_2, \cdots, e_r)$ be the sequence of Hodge slopes of $\mathcal{M}$. For every ordered $W(k)$-basis $B = (v_1, v_2, \dots, v_r)$ of $M$ and every permutation $\pi$ on $I_r = \{1, 2, \dots, r\}$, we define an $F$-crystal $\mathcal{M}_{\mathcal{E}, B, \pi} = (M, \varphi_{\mathcal{E}, B, \pi})$ by the formula $\varphi_{\mathcal{E}, B, \pi}(v_i) = p^{e_i}v_{\pi(i)}$ for all $i \in I_r$. We say that $\mathcal{M}$ with Hodge slopes $\mathcal{E}$ is \emph{$F$-cyclic} (resp. $F$-circular) if there exists an ordered $W(k)$-basis $B$ of $M$ and a permutation (resp. an $r$-cycle permutation) $\pi$ on $I_r$ such that $\mathcal{M}$ is isomorphic to $\mathcal{M}_{\mathcal{E}, B, \pi}$. When $\mathcal{E}$ and $B$ are understood, we let $\mathcal{M}_{\mathcal{E}, B, \pi} = \mathcal{M}_{\pi} = (M, \varphi_{\pi})$.
\end{definition}

Let $\mathcal{M}_{\pi} = (M, \varphi_{\pi})$ be an $F$-cyclic $F$-crystal of rank $r > 0$. For $(i,j) \in I_r^2$, let $v_{i,j}: M \to M$ be the $W(k)$-linear map such that for each $l \in I_r$, $v_{i,j}(v_l) = \delta_{j,l} v_i$. Hence $\{v_{i,j} | (i,j) \in I_r^2\}$ is a $W(k)$-basis of $\mathrm{End}_M(W(k))$, the $W(k)$-algebra of all the $W(k)$-linear endomorphisms of $M$. We denote also by $\varphi_{\pi}$ the $\sigma$-linear endomorphism of $\mathrm{End}_M(W(k))[1/p]$ given by the rule: for every $f\in \mathrm{End}_M(W(k))[1/p]$, we have $\varphi_{\pi}(f) = \varphi_{\pi} \circ f \circ \varphi^{-1}_{\pi}$. Therefore, for every $(i,j) \in I_r^2$, we have
\[\varphi_{\pi}(v_{i,j}) = p^{e_i-e_j}v_{\pi(i),\pi(j)}.\]
If we let $\epsilon_{i,j} := e_i-e_j$, then we have
\[\varphi_{\pi}(v_{i,j}) = p^{\epsilon_{i,j}}v_{\pi(i),\pi(j)}.\]
If we let $\mu_{i,j} := \mathrm{max}\{0,-\epsilon_{i,j}\}$, then for every 
\[f = \sum_{(i,j) \in I_r^2} \underline{y}_{i,j}v_{i,j} \in \mathrm{End}_M(W(k)) , \qquad \underline{y}_{i,j} \in W(k),\]  
we have 
\[\varphi_{\pi}(f)  = \sum_{(i,j) \in I_r^2} \sigma(\underline{y}_{i,j})p^{\epsilon_{i,j}}v_{\pi(i),\pi(j)} \in \mathrm{End}_M(W(k))\] 
if and only if for all $(i,j) \in I_r^2$, we have $\underline{y}_{i,j} = p^{\mu_{i,j}}\underline{x}_{i,j}$ for some $\underline{x}_{i,j} \in W(k)$.

If $\varphi_{\pi}(f) \in f + p^m\mathrm{End}_M(W(k))$, then
\begin{equation} \label{equation:breakdown1}
\sum_{(i,j) \in I_r^2} p^{\mu_{i,j}+\epsilon_{i,j}}\sigma(\underline{x}_{i,j})v_{\pi(i),\pi(j)} \equiv \sum_{(i,j) \in I_r^2}p^{\mu_{i,j}}\underline{x}_{i,j}v_{i,j} \quad \textrm{mod}\; \; p^m,
\end{equation}
whence
\begin{equation} \label{equation:breakdown2}
p^{\mu_{i,j}+\epsilon_{i,j}}\sigma(\underline{x}_{i,j})\equiv p^{\mu_{\pi(i),\pi(j)}}\underline{x}_{\pi(i),\pi(j)} \quad \textrm{mod} \; \; p^m,
\end{equation}
for all $(i,j) \in I_r^2$. Let $\mathcal{B}_{\pi \times \pi}$ be the set of orbits of $\pi \times \pi$ on $I_r^2$ and let $\mathcal{O} = \{(i_1,j_1),(i_2,j_2),\dots,(i_s,j_s)\}$ be an orbit of length $s$, i.e., $(\pi \times \pi)(i_t,j_t) = (i_{t+1},j_{t+1})$ for $t \in I_s$, where the subscripts are taken modulo $s$. For simplicity, set $x_t := x_{i_t,j_t}$, $\epsilon_t := \epsilon_{i_t,j_t} = e_{i_t}-e_{j_t}$, $\mu_t:=\mu_{i_t,j_t}$, and $\epsilon_{\mathcal{O}} := (\epsilon_1, \dots, \epsilon_s)$. Equation \eqref{equation:breakdown2} is equivalent to
\begin{equation} \label{equation:breakdown3}
p^{\mu_t+\epsilon_t}\sigma(\underline{x}_t)\equiv p^{\mu_{t+1}}\underline{x}_{t+1} \quad \textrm{mod} \; \; p^m,
\end{equation}
for all $t \in I_s$. Let $(x_{0,t}, x_{1,t}, \dots, x_{m-1,t}) \in W_m(k)$ be the reduction of $\underline{x}_t$ modulo $p^m$, where $x_{i,t} \in k$ for all $i \in \{0, 1, \dots, m-1\}$. 

\subsection{Digraphs} \label{subsection:weighteddigraphs}
Assume that $|\mathcal{O}| \geq 2$. We describe Equations \eqref{equation:breakdown3} using weighted directed graphs (or just digraphs for short in this paper), where each vertex represents a variable and the weight $w$ of each directed edge means that the equation ``$\textrm{source}^{p^{w+1}} = \textrm{target}$'' holds. We consider different possible values of $\epsilon_t$ and $ \epsilon_{t+1}$  in eleven mutually exclusive cases:

\begin{case}
\item Suppose $(\epsilon_t, \epsilon_{t+1}) \in S_{1,m}$, where 
\[S_{1,m} :=\{(x,y) \in \mathbb{Z}^2 \; | \; 0< x=-y<m\}.\]
Then \eqref{equation:breakdown3} is equivalent to
\begin{equation*} 
p^{\epsilon_t}(x_{0,t}^p,x_{1,t}^p,\dots,x_{m-1,t}^p) \equiv p^{\epsilon_t}(x_{0,t+1},x_{1,t+1},\dots,x_{m-1,t+1}) \textrm{ mod } p^m.
\end{equation*}
This is further equivalent to the system of equations\footnote{As we are only considering smooth affine group schemes over $k$ and their $k$-valued points, here and in what follows, each equation of the form $x_{i,t}^{p^a} = x_{j,t+1}^{p^b}$ is replaced by the equation $x_{i,t}^{p^{a-b}} = x_{j,t+1}$, in which $a-b$ could be negative.}:
\begin{equation} \label{equation:breakdown4}
x_{0,t}^p = x_{0,t+1},\; x_{1,t}^p = x_{1, t+1},\; \dots,\; x_{m-\epsilon_t-1, t}^p = x_{m-\epsilon_t-1,t+1}.
\end{equation} 	
We use Figure \ref{fig:case1} to represent \eqref{equation:breakdown4}.

\begin{figure}[ht]
\begin{minipage}[b]{.45\textwidth}
\centering
\begin{tikzpicture}
\tikzset{vertex/.style = {shape=circle,draw,minimum size=1.5em}}
\tikzset{edge/.style = {->,> = latex'}}
\node (a) at (0,0) {$x_{0,t}$};
\node (b) at (0,-0.5) {$x_{1,t}$};
\node (c) at (0,-1) {$\vdots$};
\node (d) at (0,-1.5) {$x_{m-\epsilon_t-1,t}$};
\node (e) at (0,-2) {$x_{m-\epsilon_t,t}$};
\node (f) at (0,-2.5) {$\vdots$};
\node (g) at (0,-3) {$x_{m-1,t}$};
\node (h) at (4,0) {$x_{0,t+1}$};
\node (i) at (4,-0.5) {$x_{1,t+1}$};
\node (j) at (4,-1) {$\vdots$};
\node (k) at (4,-1.5) {$x_{m-\epsilon_t-1,t+1}$};
\node (l) at (4,-2) {$x_{m-\epsilon_t,t+1}$};
\node (m) at (4,-2.5) {$\vdots$};
\node (n) at (4,-3) {$x_{m-1,t+1}$};
\draw[edge] (a) -- (h) node[midway, above] {\tiny $0$};
\draw[edge] (b) -- (i) node[midway, above] {\tiny $0$};
\draw[edge] (d) -- (k) node[midway, above] {\tiny $0$};
\end{tikzpicture}
\caption{Digraph for Case (1)} \label{fig:case1}
\end{minipage}
\hfill
\begin{minipage}[b]{.45\textwidth}
\centering
\begin{tikzpicture}
\tikzset{vertex/.style = {shape=circle,draw,minimum size=1.5em}}
\tikzset{edge/.style = {->,> = latex'}}
\node (a) at (0,0) {$x_{0,t}$};
\node (b) at (0,-0.5) {$x_{1,t}$};
\node (c) at (0,-1) {$\vdots$};
\node (d) at (0,-1.5) {$x_{m-\epsilon_t-1,t}$};
\node (e) at (0,-2) {$x_{m-\epsilon_t,t}$};
\node (f) at (0,-2.5) {$\vdots$};
\node (g) at (0,-3) {$x_{m-1,t}$};
\node (h) at (4,0) {$x_{0,t+1}$};
\node (i) at (4,-0.5) {$x_{1,t+1}$};
\node (j) at (4,-1) {$\vdots$};
\node (k) at (4,-1.5) {$x_{m-\epsilon_t-1,t+1}$};
\node (l) at (4,-2) {$x_{m-\epsilon_t,t+1}$};
\node (m) at (4,-2.5) {$\vdots$};
\node (n) at (4,-3) {$x_{m-1,t+1}$};
\draw[edge] (a) -- (h) node[midway, above] {\tiny $0$};
\draw[edge] (b) -- (i) node[midway, above] {\tiny $0$};
\draw[edge] (d) -- (k) node[midway, above] {\tiny $0$};
\draw[edge] (e) -- (l) node[midway, above] {\tiny $0$};
\draw[edge] (g) -- (n) node[midway, above] {\tiny $0$};
\end{tikzpicture}
\caption{Digraph for Case (2)} \label{fig:case2}
\end{minipage}
\end{figure}

\item Suppose $(\epsilon_t, \epsilon_{t+1}) \in S_{2}$, where 
\[S_{2} = S_{2,m} := \{(x,y) \in \mathbb{Z}^2 \; | \; x \leq 0 \leq y\}.\] 
Then \eqref{equation:breakdown3} is equivalent to
\begin{equation*} 
(x_{0,t}^p,x_{1,t}^p,\dots,x_{m-1,t}^p) \equiv (x_{0,t+1},x_{1,t+1},\dots,x_{m-1,t+1}) \textrm{ mod } p^m.
\end{equation*}
This is further equivalent to the system of equations:
\begin{equation} \label{equation:breakdown7}
x_{0,t}^p = x_{0,t+1},\; x_{1,t}^p = x_{1,t+1},\; \dots,\; x_{m-1,t}^p = x_{m-1,t+1}.
\end{equation}
We use Figure \ref{fig:case2} to represent \eqref{equation:breakdown7}.

\item Suppose $(\epsilon_t, \epsilon_{t+1}) \in S_{3,m}$, where 
\[S_{3,m} :=\{(x,y) \in \mathbb{Z}^2 \; | \; 0 \leq x<-y < m\}.\] 
Then \eqref{equation:breakdown3} is equivalent to
\begin{equation*} 
p^{\epsilon_s}(x_{0,t}^p,x_{1,t}^p,\dots,x_{m-1,t}^p) \equiv p^{-\epsilon_{t+1}}(x_{0,t+1},x_{1,t+1},\dots,x_{m-1,t+1}) \textrm{ mod } p^m.
\end{equation*}
This is further equivalent to the system of equations:
\begin{equation} \label{equation:breakdown5}
\begin{aligned}
x_{0,t} = x_{1,t} = & \cdots = x_{-\epsilon_{t+1}-\epsilon_t-1,t} = 0, \\
x_{-\epsilon_{t+1}-\epsilon_t,t}^{p^{\epsilon_{t+1}+\epsilon_t+1}} = x_{0,t+1},\; & \dots,\; x_{m-\epsilon_t-1,t}^{p^{\epsilon_{t+1}+\epsilon_t+1}} = x_{m+\epsilon_{t+1}-1,t+1}.
\end{aligned}
\end{equation}
We use Figure \ref{fig:case3} to represent \eqref{equation:breakdown5}.

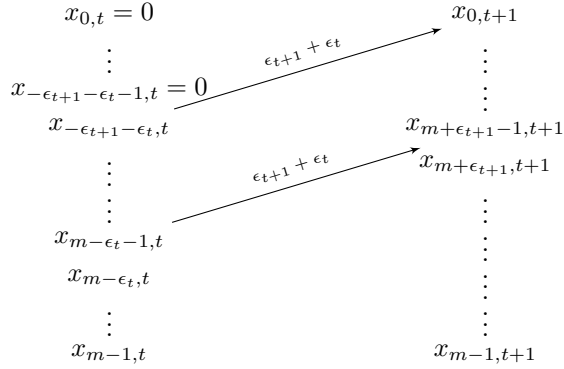
\begin{figure}
\centering
\begin{tikzpicture}
\tikzset{vertex/.style = {shape=circle,draw,minimum size=1.5em}}
\tikzset{edge/.style = {->,> = latex'}}
\node (a) at (0,0) {$x_{0,t}=0$};
\node (b) at (0,-0.5) {$\vdots$};
\node (c) at (0,-1) {$x_{-\epsilon_{t+1}-\epsilon_t-1,t}=0$};
\node (d) at (0,-1.5) {$x_{-\epsilon_{t+1}-\epsilon_t,t}$};
\node (e) at (0,-2) {$\vdots$};
\node (f) at (0,-2.5) {$\vdots$};
\node (g) at (0,-3) {$x_{m-\epsilon_t-1,t}$};
\node (h) at (0,-3.5) {$x_{m-\epsilon_t,t}$};
\node (i) at (0,-4) {$\vdots$};
\node (j) at (0,-4.5) {$x_{m-1,t}$};
\node (k) at (5,0) {$x_{0,t+1}$};
\node (l) at (5,-0.5) {$\vdots$};
\node (m) at (5,-1) {$\vdots$};
\node (n) at (5,-1.5) {$x_{m+\epsilon_{t+1}-1,t+1}$};
\node (o) at (5,-2) {$x_{m+\epsilon_{t+1},t+1}$};
\node (p) at (5,-2.5) {$\vdots$};
\node (q) at (5,-3) {$\vdots$};
\node (r) at (5,-3.5) {$\vdots$};
\node (s) at (5,-4) {$\vdots$};
\node (t) at (5,-4.5) {$x_{m-1,t+1}$};
\draw[edge] (d) -- (k) node[sloped,midway,anchor=center, above] {\tiny $\epsilon_{t+1}+\epsilon_t$};
\draw[edge] (g) -- (n) node[sloped,midway,anchor=center, above] {\tiny $\epsilon_{t+1}+\epsilon_t$};
\end{tikzpicture}
\caption{Digraph for Case (3) with negative weights} \label{fig:case3}
\end{figure}

\item Suppose $(\epsilon_t, \epsilon_{t+1}) \in S_{4,m}$, where 
\[S_{4,m} :=\{(x,y) \in \mathbb{Z}^2 \; | \; 0 \leq x < m \leq -y\}.\] 
Then \eqref{equation:breakdown3} is equivalent to the system of equations:
\begin{equation} \label{equation:breakdown12}
x_{0,t} = \cdots = x_{m-\epsilon_t-1,t} = 0.
\end{equation}
The digraph representing \eqref{equation:breakdown12} contains no edges and two columns of vertices, i.e., $x_{i,t}$ on the left and $x_{i,t+1}$ on the right for $i \in \{0,\dots,m-1\}$, with the first $m-\epsilon_t$ top left vertices marked to be $0$.

\item Suppose $(\epsilon_t, \epsilon_{t+1}) \in S_{5,m}$, where 
\[S_{5,m} :=\{(x,y) \in \mathbb{Z}^2 \; | \; 0 \leq -y < x < m\}.\] 
Then \eqref{equation:breakdown3} is, as in Case (3), equivalent to 
\begin{equation*} 
p^{\epsilon_s}(x_{0,t}^p,x_{1,t}^p,\dots,x_{m-1,t}^p) \equiv p^{-\epsilon_{t+1}}(x_{0,t+1},x_{1,t+1},\dots,x_{m-1,t+1}) \textrm{ mod } p^m.
\end{equation*}
This is further equivalent to the system of equations:
\begin{equation} \label{equation:breakdown6}
\begin{aligned}
x_{0,t+1} = x_{1,t+1} = & \cdots = x_{\epsilon_t+\epsilon_{t+1}-1,t+1} = 0,\\
x_{0,t}^{p^{\epsilon_{t+1}+\epsilon_t+1}} = x_{\epsilon_{t+1}+\epsilon_t,t+1},\; & \dots,\; x_{m-\epsilon_t-1,t}^{p^{\epsilon_{t+1}+\epsilon_t+1}} = x_{m+\epsilon_{t+1}-1,t+1}.
\end{aligned}
\end{equation}
We use Figure \ref{fig:case4} to represent \eqref{equation:breakdown6}.

\begin{figure}
\centering
\begin{tikzpicture}
\tikzset{vertex/.style = {shape=circle,draw,minimum size=1.5em}}
\tikzset{edge/.style = {->,> = latex'}}
\node (a) at (0,0) {$x_{0,t}$};
\node (b) at (0,-0.5) {$\vdots$};
\node (c) at (0,-1) {$\vdots$};
\node (d) at (0,-1.5) {$x_{m-\epsilon_t-1,t}$};
\node (e) at (0,-2) {$x_{m-\epsilon_t,t}$};
\node (f) at (0,-2.5) {$\vdots$};
\node (g) at (0,-3) {$\vdots$};
\node (h) at (0,-3.5) {$\vdots$};
\node (i) at (0,-4) {$\vdots$};
\node (j) at (0,-4.5) {$x_{m-1,t}$};
\node (k) at (5,0) {$x_{0,t+1}=0$};
\node (l) at (5,-0.5) {$\vdots$};
\node (m) at (5,-1) {$x_{\epsilon_t+\epsilon_{t+1}-1,t+1}=0$};
\node (n) at (5,-1.5) {$x_{\epsilon_t+\epsilon_{t+1},t+1}$};
\node (o) at (5,-2) {$\vdots$};
\node (p) at (5,-2.5) {$\vdots$};
\node (q) at (5,-3) {$x_{m+\epsilon_{t+1}-1,t+1}$};
\node (r) at (5,-3.5) {$x_{m+\epsilon_{t+1},t+1}$};
\node (s) at (5,-4) {$\vdots$};
\node (t) at (5,-4.5) {$x_{m-1,t+1}$};
\draw[edge] (a) -- (n) node[sloped,midway,anchor=center, above] {\tiny $\epsilon_{t+1}+\epsilon_t$};
\draw[edge] (d) -- (q) node[sloped,midway,anchor=center, above] {\tiny $\epsilon_{t+1}+\epsilon_t$};
\end{tikzpicture}
\caption{Digraph for Case (5) with positive weights} \label{fig:case4}
\end{figure}
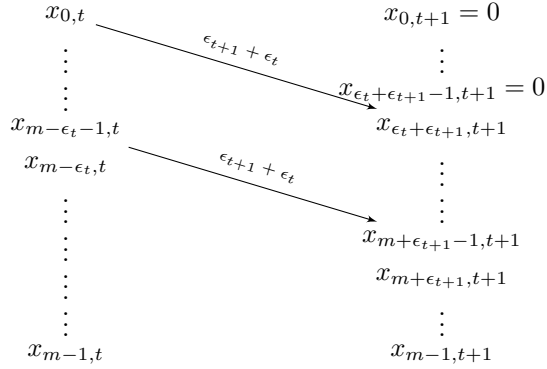

\item Suppose $(\epsilon_t, \epsilon_{t+1}) \in S_{6,m}$, where 
\[S_{6,m} :=\{(x,y) \in \mathbb{Z}^2 \; | \; 0 \leq -y < m \leq x \}.\] 
Then \eqref{equation:breakdown3} is equivalent to the system of equations:
\begin{equation} \label{equation:breakdown13}
x_{0,t+1} = \cdots = x_{m+\epsilon_{t+1}-1,t+1} = 0.
\end{equation}
The digraph representing \eqref{equation:breakdown13} contains no edges and two columns of vertices, i.e., $x_{i,t}$ on the left and $x_{i,t+1}$ on the right for $i \in \{0, \dots, m-1\}$, with the first $m+\epsilon_{t+1}$ top right vertices marked to be $0$.

\item Suppose $(\epsilon_t, \epsilon_{t+1}) \in S_{7,m}$, where \[S_{7,m} :=\{(x,y) \in \mathbb{Z}^2 \; | \; x \geq m, \; y \leq -m\}.\] 
In this case, \eqref{equation:breakdown3} always holds and hence the digraph contains no edges and just two columns of vertices, i.e., $x_{i,t}$ on the left and $x_{i,t+1}$ on the right for $i \in \{0, \dots, m-1\}$ (so no vertices are marked to be $0$).

\item Suppose $(\epsilon_t, \epsilon_{t+1}) \in S_{8,m}$, where \[S_{8,m} := \{(x,y) \in \mathbb{Z}^2 \; | \; x < 0 < -y < m\}.\]
Then \eqref{equation:breakdown3} is equivalent to
\begin{equation*} 
(x_{0,t}^p,x_{1,t}^p,\dots,x_{m-1,t}^p) \equiv p^{-\epsilon_{t+1}}(x_{0,t+1},x_{1,t+1},\dots,x_{m-1,t+1}) \textrm{ mod } p^m.
\end{equation*}
This is further equivalent to the system of equations:
\begin{equation} \label{equation:breakdown8}
\begin{aligned}
x_{0,t} = x_{1,t} = &\cdots = x_{-\epsilon_{t+1}-1,t} = 0, \\
x_{-\epsilon_{t+1},t}^{p^{\epsilon_{t+1}+1}} = x_{0,t+1},\; & \dots,\; x_{m-1,t}^{p^{\epsilon_{t+1}+1}} = x_{m+\epsilon_{t+1}-1,t+1}.
\end{aligned}
\end{equation}
We use Figure \ref{fig:case5} to represent \eqref{equation:breakdown8}.

\begin{figure}
\centering
\begin{tikzpicture}
\tikzset{vertex/.style = {shape=circle,draw,minimum size=1.5em}}
\tikzset{edge/.style = {->,> = latex'}}
\node (a) at (0,0) {$x_{0,t}=0$};
\node (b) at (0,-0.5) {$\vdots$};
\node at (0,-1) {$x_{-\epsilon_{t+1}-1,t}=0$};
\node (c) at (0,-1.5) {$x_{-\epsilon_{t+1},t}$};
\node (d) at (0,-2) {$\vdots$};
\node (e) at (0,-2.5) {$x_{m-1,t}$};
\node (f) at (5,0) {$x_{0,t+1}$};
\node (g) at (5,-0.5) {$\vdots$};
\node (h) at (5,-1) {$x_{m+\epsilon_{t+1}-1,t+1}$};
\node at (5,-1.5) {$x_{m+\epsilon_{t+1},t+1}$};
\node (i) at (5,-2) {$\vdots$};
\node (j) at (5,-2.5) {$x_{m-1,t+1}$};
\draw[edge] (c) -- (f) node[sloped,midway,anchor=center, above] {\tiny $\epsilon_{t+1}$};
\draw[edge] (e) -- (h) node[sloped,midway,anchor=center, above] {\tiny $\epsilon_{t+1}$};
\end{tikzpicture}
\caption{Digraph for Case (8) with negative weights} \label{fig:case5}
\end{figure}
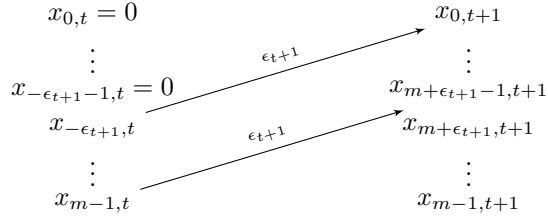

\item Suppose $(\epsilon_t, \epsilon_{t+1}) \in S_{9,m}$, where \[S_{9,m} := \{(x,y) \in \mathbb{Z}^2 \; | \; x < 0 < m \leq -y\}.\] 
Then \eqref{equation:breakdown3} is equivalent to the system of equations:
\begin{equation} \label{equation:breakdown10}
x_{0	,t} = x_{1,t} = \cdots  = x_{m-1,t} = 0.
\end{equation}
The digraph representing \eqref{equation:breakdown10} contains no edges, $m$ zero vertices on the left, and $x_{i,t+1}$ for $i \in \{0, \dots, m-1\}$ as vertices on the right. 

\item Suppose $(\epsilon_t, \epsilon_{t+1}) \in S_{10,m}$, where \[S_{10,m} := \{(x,y) \in \mathbb{Z}^2 \; | \; 0 < x < m, \; y > 0 \}.\]
Then \eqref{equation:breakdown3} is equivalent to
\begin{equation*} 
p^{\epsilon_t}(x_{1,t}^p,x_{2,t}^p,\dots,x_{m,t}^p) \equiv (x_{1,t+1},x_{2,t+1},\dots,x_{m,t+1}) \quad \textrm{mod} \; \; p^m.
\end{equation*}
This is further equivalent to the system of equations:
\begin{equation} \label{equation:breakdown9}
\begin{aligned}
x_{0,t+1} = x_{1,t+1} = & \cdots = x_{\epsilon_t-1,t+1} = 0, \\
x_{0,t}^{p^{\epsilon_t+1}} = x_{\epsilon_t,t+1},\; & \dots,\; x_{m-\epsilon_t-1,t}^{p^{\epsilon_t+1}} = x_{m-1,t+1}.
\end{aligned}
\end{equation}
We use Figure \ref{fig:case6} to represent \eqref{equation:breakdown9}.

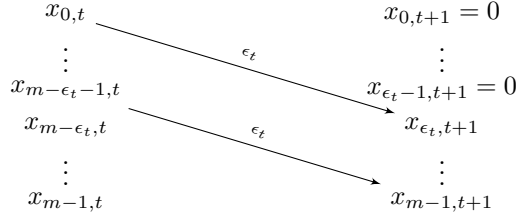
\begin{figure}
\centering
\begin{tikzpicture}
\tikzset{vertex/.style = {shape=circle,draw,minimum size=1.5em}}
\tikzset{edge/.style = {->,> = latex'}}
\node (a) at (0,0) {$x_{0,t}$};
\node (b) at (0,-0.5) {$\vdots$};
\node (c) at (0,-1) {$x_{m-\epsilon_t-1,t}$};
\node at (0,-1.5) {$x_{m-\epsilon_t,t}$};
\node (d) at (0,-2) {$\vdots$};
\node (e) at (0,-2.5) {$x_{m-1,t}$};
\node (f) at (5,0) {$x_{0,t+1}=0$};
\node (g) at (5,-0.5) {$\vdots$};
\node at (5,-1) {$x_{\epsilon_t-1,t+1}=0$};
\node (h) at (5,-1.5) {$x_{\epsilon_t,t+1}$};
\node (i) at (5,-2) {$\vdots$};
\node (j) at (5,-2.5) {$x_{m-1,t+1}$};
\draw[edge] (a) -- (h) node[sloped,midway,anchor=center, above] {\tiny $\epsilon_t$};
\draw[edge] (c) -- (j) node[sloped,midway,anchor=center, above] {\tiny $\epsilon_t$};
\end{tikzpicture}
\caption{Digraph for Case (10) with positive weights} \label{fig:case6}
\end{figure}

\item Suppose $(\epsilon_t, \epsilon_{t+1}) \in S_{11,m}$, where \[S_{11,m} := \{(x,y) \in \mathbb{Z}^2 \; | \; x \geq m, \; y > 0 \}.\]
Then \eqref{equation:breakdown3} is equivalent to the system of equations:
\begin{equation} \label{equation:breakdown11}
x_{0,t+1} = x_{1,t+1} = \cdots  = x_{m-1,t+1} = 0.
\end{equation}
The digraph representing \eqref{equation:breakdown11} contains no edges, $m$ zero vertices on the right, and $x_{i,t}$ for $i \in \{0,\dots,m-1\}$ as vertices on the left.
\end{case}
The above eleven cases cover all possible values of $(\epsilon_t, \epsilon_{t+1})$ as 
\[\mathbb{Z}^2 = \bigsqcup_{i=1}^{11} S_{i,m}.\] 
If $x_{i,t}x_{j,t+1}$ is an edge in any one of the eleven digraphs above, then its weight only depends on the values of $\epsilon_t$ and $\epsilon_{t+1}$ as it is equal to $\textrm{max}(\epsilon_t, 0) + \textrm{min}(\epsilon_{t+1},0)$. Furthermore, we note that the weight of $x_{i,t}x_{j,t+1}$ is equal to $j-i$. Hence we have
\begin{equation} \label{equation:changeofweight}
\mathrm{lv}(\epsilon_t, \epsilon_{t+1}) := \textrm{max}(\epsilon_t, 0) + \textrm{min}(\epsilon_{t+1},0) = j-i.
\end{equation}

\begin{remark}
If $(\epsilon_t, \epsilon_{t+1}) \in \mathbb{Z}^2$ with $|\epsilon_t| > m$ (resp. $|\epsilon_{t+1}| > m$), then the constructed digraph does not change if we replace $\epsilon_t$ (resp. $\epsilon_{t+1}$) by $\text{sgn}(\epsilon_t) \cdot m$ (resp. $\text{sgn}(\epsilon_{t+1}) \cdot m$).
\end{remark}

\begin{remark} \label{remark:digraph1}
If $|\mathcal{O}| = 1$, then \eqref{equation:breakdown3} is equivalent to
\begin{equation} \label{equation:breakdown14}
p^{\mu_1+\epsilon_1}\sigma(\underline{x}_1) = p^{\mu_1}\underline{x}_1 \textrm{ mod } p^m.
\end{equation}
If $\epsilon_1 > 0$ (resp. $\epsilon_1<0$), then \eqref{equation:breakdown14} is equivalent to $p^{\epsilon_1}\sigma(\underline{x}_1) = \underline{x}_1$ mod $p^m$ (resp. $\sigma(\underline{x}_1) = p^{\mu_1}\underline{x}_1$ mod $p^m$). In either case, we know that
\begin{equation} \label{equation:breakdown16}
x_{0,1} = x_{1,1} = \cdots = x_{m-1,t} = 0.
\end{equation}
The digraph representing \eqref{equation:breakdown16} contains $m$ zero vertices and no edges.

If $\epsilon_1 = 0$, then \eqref{equation:breakdown14} is equivalent to
\[\sigma(\underline{x}_1) = \underline{x}_1 \textrm{ mod } p^m.\]
This is equivalent to
\begin{equation} \label{equation:breakdown15}
x_{0,1}^p = x_{0,1}, \dots, x_{m-1,t}^p = x_{m-1,t}.
\end{equation}
The digraph representing \eqref{equation:breakdown15} contains $m$ vertices, i.e. $x_{0,1}, \dots, x_{m-1,1}$. For each $i \in \{0, \dots, m-1\}$, there is an edge whose source and target are both $x_{i,1}$ and whose weight is $0$.
\end{remark}

To end this subsection, we construct the digraph $\Gamma_{\mathbf{End}_m(\mathcal{M}_{\pi})}$ as follows. For each $\mathcal{O} \in \mathcal{B}_{\pi \times \pi}$, we have a circular sequence of integer $\epsilon_{\mathcal{O}} = (\epsilon_1, \dots, \epsilon_{|\mathcal{O}|})$ such that for each $t \in I_{|\mathcal{O}|}$, $|\epsilon_t|$ is less than or equal to the difference of the largest and the smallest Hodge slope of $\mathcal{M}_{\pi}$. If $|\mathcal{O}| = 1$, then $\Gamma^m_{\epsilon_{\mathcal{O}}}$ is given by Remark \ref{remark:digraph1}. Suppose $|\mathcal{O}| \geq 2$. For each pair $(\epsilon_t, \epsilon_{t+1})$, we have constructed a digraph $\Gamma^m_{\epsilon_t,\epsilon_{t+1}} = (V^m_{\epsilon_t, \epsilon_{t+1}}, E^m_{\epsilon_t, \epsilon_{t+1}}, w^m_{\epsilon_t, \epsilon_{t+1}})$, where $V^m_{\epsilon_t, \epsilon_{t+1}}$ is the vertex set, $E^m_{\epsilon_t, \epsilon_{t+1}}$ is the edge set, and $w^m_{\epsilon_t, \epsilon_{t+1}} : E^m_{\epsilon_t, \epsilon_{t+1}} \to \mathbb{Z}$ is the weight function. The vertex set $V^m_{\epsilon_t, \epsilon_{t+1}}$ has $2m$ vertices. We define $\Gamma^m_{\epsilon_{\mathcal{O}}}$ to be the union of all $\Gamma^m_{\epsilon_t,\epsilon_{t+1}}$ as $t \in I_{|\mathcal{O}|}$ with the understanding that $\epsilon_{|\mathcal{O}|+1} := \epsilon_1$. If $|\mathcal{O}| = 2$, then $V^m_{\epsilon_1, \epsilon_2} = V^m_{\epsilon_2,\epsilon_1}$ and thus the union of the vertex sets is a nondisjoint union. If $|\mathcal{O}| \geq 3$, then the union of the vertex sets is also a nondisjoint union as $V^m_{\epsilon_{t-1},\epsilon_t} \cap V^m_{\epsilon_t, \epsilon_{t+1}} = \{x_{0,t}, \dots, x_{m-1,t}\}$. The edge set of $\Gamma^m_{\epsilon_{\mathcal{O}}}$ is a disjoint union as $E^m_{\epsilon_{t-1},\epsilon_t} \cap E^m_{\epsilon_t, \epsilon_{t+1}} = \emptyset$. The weight function of $\Gamma^m_{\epsilon_{\mathcal{O}}}$ is well-defined as the edge set  of $\Gamma^m_{\epsilon_{\mathcal{O}}}$ is a disjoint union. We define 
\[\Gamma_{\mathbf{End}_m(\mathcal{M})} = \bigsqcup_{\mathcal{O} \in \mathcal{B}_{\pi \times \pi}} \Gamma^m_{\epsilon_{\mathcal{O}}}.\]

\subsection{Connected Components and Dimension of $\mathbf{End}_m(\mathcal{M})$}
We recall two basic definitions from graph theory. By a circular graph, we mean a connected graph that has exactly one cycle and that every vertex has degree $2$ \footnote{In a circular graph with one vertex $v$ and one edge, the degree of $v$ is still $2$.}. A (weighted) circular digraph is a weighted circular directed graph. By a linear graph, we mean a connected graph that either has one vertex with no edges or has at least two vertices of degree $1$ while other vertices (if any) have degree $2$. A (weighted) linear digraph is a weighted linear directed graph. 

We want to understand how zero vertices affect other vertices in $\Gamma_{\mathbf{End}_m(\mathcal{M})}$. During the construction of $\Gamma_{\mathbf{End}_m(\mathcal{M})}$, each vertex that has already been regarded as zero has degree either $0$ or $1$, and thus cannot be in a circular digraph. If a vertex of degree $1$ has been regarded as zero, then it must be either the origin or the terminal of a linear digraph. Due to the relation $\textrm{source}^{p^{w+1}} = \textrm{target}$, all other vertices in that linear digraph will be regarded as zero vertices as well. In other words, if there is one vertex that has been regarded as zero in a linear digraph, then all vertices in that linear digraph will be regarded as zero.

\begin{definition}
We call a (weighted) linear digraph that does not contain any zero vertex a (weighted) \emph{free linear digraph}.
\end{definition}

There are two types of connected components of $\Gamma_{\mathbf{End}_m(\mathcal{M})}$ which have no vertex regarded as zero: free linear digraphs and circular digraphs. We denote by $\ell(\Gamma_{\mathbf{End}_m(\mathcal{M})})$ the number of free linear digraphs in $\Gamma_{\mathbf{End}_m(\mathcal{M})}$. For each edge $x_{i,t}x_{j,t+1}$ of a circular digraph of weight $w$, we recall that $w = j-i$. As the origin and the terminal of a circular digraph are the same, we know that the sum of weights of all edges of every circular digraph is equal to $0$. We denote by $w(\Gamma_{\mathbf{End}_m(\mathcal{M})})$ the total number of edges in all circular digraphs in $\Gamma_{\mathbf{End}_m(\mathcal{M})}$.

\begin{proposition} \label{proposition:firstcalc}
Let $\mathcal{M} \cong \mathcal{M}_{\pi}$ be an $F$-cyclic $F$-crystal over $k$.
\begin{enumerate}
\item The number of connected components of $\mathbf{End}_m(\mathcal{M})$ is equal to $p^{w(\Gamma_{\mathbf{End}_m(\mathcal{M})})}$.
\item  The dimension of $\mathbf{End}_m(\mathcal{M})$ is equal to $\ell(\Gamma_{\mathbf{End}_m(\mathcal{M})})$.
\end{enumerate} 
\end{proposition}
\begin{proof}
For any $k$-algebra $R$, let $R^{\textrm{perf}}$ be the perfection of $R$. If $R$ is smooth, let $c(R) = c(R^\textrm{perf})$ be the number of connected components of $\mathrm{Spec}\,R$ or $\mathrm{Spec}\,R^{\textrm{perf}}$. Let $r$ be the rank of $\mathcal{M}$ and let $\mathbf{E} := \mathbf{End}_m(\mathcal{M})$. Let
\[X := \{x_{t,i,j} \mid t = 0, 1, \dots, m-1, (i,j) \in I_r^2\},\]
and $\mathfrak{I}$ be the ideal of $k[X]$ generated by the relations defined by \eqref{equation:breakdown2} for all $(i,j) \in I_r^2$, where $\underline{x}_{i,j} \equiv (x_{0,i,j}, \dots, x_{m-1,i,j}) \in W_m(k)$ modulo $p^m$. For every $l \in I_m$ and $(i,j) \in I_r^2$, if $x_{l,i,j}$ is regarded as a zero vertex in $\Gamma_{\mathbf{E}}$, then $x_{l,i,j} \in \mathfrak{I}$. The perfection of the representing $k$-algebra of $\mathbf{E}$ is 
\[k[\mathbf{E}]^{\textrm{perf}} \cong k[X]^{\textrm{perf}}/\sqrt{\mathfrak{I}}.\]
For every orbit $\mathcal{O} \in B_{\pi \times \pi}$, let
\[X_{\mathcal{O}} := \{x_{t,i,j} \mid t = 0, 1, \dots, m-1, (i,j) \in \mathcal{O}\},\]
and $\mathfrak{I}_{\mathcal{O}}$ be the ideal of $k[X_{\mathcal{O}}]$ generated by the relations defined by \eqref{equation:breakdown2} for all $(i,j) \in \mathcal{O}$. Let
\[k[\mathcal{O}] := k[X_{\mathcal{O}}]^{\textrm{perf}}/\sqrt{\mathfrak{I}_{\mathcal{O}}},\]
and thus we have
\[k[\mathbf{E}]^{\textrm{perf}} \cong \prod_{\mathcal{O} \in B_{\pi \times \pi}} k[\mathcal{O}].\]
For each $\mathcal{O} \in B_{\pi \times \pi}$, let $C_{\mathcal{O}}$ be the set of connected components of $\Gamma^m_{\epsilon_{\mathcal{O}}}$. For each $\mathcal{C} \in C_{\mathcal{O}}$, let $V_{\mathcal{C}}$ be the set of vertices of $\mathcal{C}$ and $E_{\mathcal{C}}$ be the set of edges of $\mathcal{C}$. Let $\mathfrak{I}_{\mathcal{C}}$ be the ideal of $k[V_{\mathcal{C}}]$ generated by the relations defined by the edge set $E_{\mathcal{C}}$. Let
\[k[\mathcal{C}] := k[V_{\mathcal{C}}]^{\text{perf}}/\sqrt{\mathfrak{I}_{\mathcal{C}}},\]
and thus we have
\begin{equation} \label{equation:decompositionrepresentingalgebra}
k[\mathbf{E}]^{\textrm{perf}} \cong \prod_{\mathcal{O} \in B_{\pi \times \pi}} k[\mathcal{O}] \cong \prod_{\mathcal{O} \in B_{\pi \times \pi}} \prod_{\mathcal{C} \in C_{\mathcal{O}}} k[\mathcal{C}].
\end{equation}

The first type of connected components $\mathcal{C}_1 \in C_{\mathcal{O}}$ is a linear digraph. If $\mathcal{C}_1$ contains a zero vertex, then the entire linear digraph contains just zero vertices. Hence $k[\mathcal{C}_1] \cong k$ has dimension $0$ and has exactly one connected component. Suppose $\mathcal{C}_1$ does not contain a zero vertex (and thus it is a free linear digraph) and the vertices (in order) are $x_{l_1,1}, x_{l_2,2}, \dots, x_{l_t,t}$, where $x_{l_1,1}$ is the origin and $x_{l_t,t}$ is the terminal. We have $x_{l_{j+1},j+1} = x_{l_j,j}^{p^{w_j}}$, where $w_j$ is the weight of the edge $x_{l_j,j}x_{l_{j+1},j+1}$ for all $j \in I_{t-1}$. In this case, $x_{l_1,1}$ can be considered as a free variable in $k[\mathcal{C}_1]$ and $x_{l_2,2}, \dots, x_{l_t,t}$ depend on the choice of $x_{l_1,1}$. We easily get that  $k[\mathcal{C}_1] \cong k[x_{l_1,1}]^{\text{perf}}$ has dimension $1$ and exactly one connected component.

The second type of connected components $\mathcal{C}_2 \in C_{\mathcal{O}}$ is a circular digraph. Recall that no vertices in a circular digraph can be zero because zero vertices have degree either $0$ or $1$. Let $w = w(\mathcal{C}_2)$ be the number of edges in $\mathcal{C}_2$. Let $x_{l,1}$ be a vertex of $\mathcal{C}_2$. Because the sum of all weights of edges in $\mathcal{C}_2$ is equal to $0$, we have $x_{l,1} = x_{l,1}^{p^w}$. There are $p^w$ solutions to this equation in $k$ and once the value of $x_{l,1}$ is fixed (among these $p^w$ solutions), the values of each vertex of $\mathcal{C}_2$ is uniquely determined. Hence $k[\mathcal{C}_2] \cong k^{p^w}$ has dimension $0$ and exactly $p^w$ connected components.

Based on \eqref{equation:decompositionrepresentingalgebra}, we conclude that the dimension of $\mathbf{E}$ is equal to
\[\textrm{dim}(k[\mathbf{E}]^{\textrm{perf}}) = \sum_{\mathcal{O} \in B_{\pi \times \pi}} \sum_{\mathcal{C} \in C_{\mathcal{O}}} \textrm{dim}(k[\mathcal{C}]) = \sum_{\substack{\textrm{free linear digraph} \\ \mathcal{C}_1 \textrm{ in } \Gamma_{\mathbf{E}}}} 1 = \ell(\Gamma_{\mathbf{E}}).\]
Similarly, the number of connected components of $\mathbf{E}$ is equal to 
\begin{align*}
c(k[\mathbf{E}]^{\textrm{perf}}) &= \prod_{\mathcal{O} \in B_{\pi \times \pi}} \prod_{\mathcal{C} \in C_{\mathcal{O}}} c(k[\mathcal{C}]) = \prod_{\substack{\textrm{circular digraph} \\ \mathcal{C}_2 \textrm{ in } \Gamma_{\mathbf{E}}}} c(k[\mathcal{C}_2]) \\
&= \prod_{\substack{\textrm{circular digraph} \\ \mathcal{C}_2 \textrm{ in } \Gamma_{\mathbf{E}}}} p^{w(\mathcal{C}_2)} = p^{w(\Gamma_{\mathbf{E}})} \qedhere.
\end{align*}
\end{proof}

\section{Two Reductions} \label{tworeductions}

In this section, we first generalize the construction of weighted digraphs in Subsection \ref{subsection:weighteddigraphs} to any circular sequence of integers $\epsilon$. Then we develop two reductions that allow us to modify the sequence $\epsilon$ without changing the number of free linear digraphs and the number of circular digraphs in the associated weighted digraph. Using these two reductions, we will able to get combinatorial formulas for the dimension and the number of connected components of $\mathbf{End}_m(\mathcal{M})$ in the next section.

\subsection{Generalization}
\begin{definition}
Let $\epsilon = (\epsilon_1, \epsilon_2, \dots, \epsilon_s)$ be a circular sequence of integers. If $s \geq 2$, for every $t \in I_s$, there is a digraph $\Gamma^m_{\epsilon_t, \epsilon_{t+1}}$ (as constructed in Subsection \ref{subsection:weighteddigraphs}) associated to the pair of integers $(\epsilon_t, \epsilon_{t+1})$. Here $\epsilon_{s+1} := \epsilon_1$. The vertex set of $\Gamma^m_{\epsilon_t, \epsilon_{t+1}}$ is 
\[\{x_{0,t}, \dots, x_{m-1,t}, x_{0,t+1}, \dots, x_{m-1,t+1}\},\]
and each edge (if any) of $\Gamma^m_{\epsilon_t, \epsilon_{t+1}}$ has source $x_{i,t}$ and target $x_{j,t+1}$ for some $i, j \in \{0, \dots, m-1\}$. Let $\Gamma^m_{\epsilon}$ be the nondisjoint union of $\Gamma^m_{\epsilon_t, \epsilon_{t+1}}$ for all $i \in I_s$. If $s = 1$, then let $\Gamma^m_{\epsilon}$ be the digraph constructed in the same way as in Remark \ref{remark:digraph1}. Let $\ell(\Gamma^m_{\epsilon})$ be the number of free linear digraphs in $\Gamma^m_{\epsilon}$, $c(\Gamma^m_{\epsilon})$ be the number of circular digraphs in $\Gamma^m_{\epsilon}$, and $w(\Gamma^m_{\epsilon})$ be the total number of edges in all circular digraphs in $\Gamma^m_{\epsilon}$.
\end{definition}

\begin{remark} \label{remark:negativereverse}
Let
\[\eta = (\eta_s, \eta_{s-1}, \dots, \eta_1) := (-\epsilon_s, -\epsilon_{s-1}, \dots, -\epsilon_1).\]
For every pair of consecutive integers $(\eta_{t+1}, \eta_{t})$ in $\eta$, let $\Gamma^m_{\eta_{t+1}, \eta_t}$ be the digraph (as in Subsection \ref{subsection:weighteddigraphs}) with the vertex set  
\[\{x_{0,t+1}, \dots, x_{m-1,t+1}, x_{0,t}, \dots, x_{m-1,t}\},\]
and each edge (if any) of $\Gamma^m_{\eta_{t+1}, \eta_t}$ has source $x_{j,t+1}$ and target $x_{i,t}$ for some $i, j \in \{0, \dots, m-1\}$. Let $\Gamma^m_{\eta}$ be the (nondisjoint if $s \geq 2$) union of $\Gamma^m_{\eta_{t+1}, \eta_{t}}$. 

The digraph $\Gamma^m_{\eta_{t+1}, \eta_t}$ can be obtained from $\Gamma^m_{\epsilon_t, \epsilon_{t+1}}$ by first reversing the direction of all edges (if any), and then taking additive inverses of all the weights. As a result, $\Gamma^m_{\eta}$ can be obtained from $\Gamma^m_{\epsilon}$ by first reversing the direction of all edges (if any), and then taking additive inverses of all the weights.
\end{remark}

\begin{proposition} \label{proposition:circularzero}
For every circular sequence of integers $\epsilon = (\epsilon_1, \epsilon_2, \dots, \epsilon_s)$, if $\Gamma^m_{\epsilon}$ contains a circular digraph, then each circular digraph of $\Gamma^m_{\epsilon}$ contains exactly $s$ vertices and $\sum_{i=1}^s \epsilon_i = 0$.
\end{proposition}
\begin{proof}
If $\Gamma^m_{\epsilon}$ contains a circular digraph, then the number of vertices of the circular digraph is a multiple of $s$ because the origin and the terminal of a circular digraph must be the same. Suppose that there is a circular digraph that contains $ts$ vertices. Let 
\[x_{l_{11},1}, x_{l_{12}, 2}, \dots, x_{l_{1s},s}, x_{l_{21},1}, x_{l_{22},2}, \dots, x_{l_{2s},s}, \dots,  x_{l_{t1},1}, x_{l_{t2},2}, \dots, x_{l_{ts},s}\]
denote the sequence of its vertices. By considering the difference
\begin{align*}
l_{(j+1)1} - l_{j1} &= (l_{j2} - l_{j1}) + (l_{j3}-l_{j2}) + \cdots + (l_{js} - l_{j(s-1)}) + (l_{(j+1)1} - l_{js}), \\
&= \mathrm{lv}(\epsilon_1,\epsilon_2)+ \mathrm{lv}(\epsilon_2,\epsilon_3) + \cdots + \mathrm{lv}(\epsilon_{s-1},\epsilon_s) + \mathrm{lv}(\epsilon_s,\epsilon_1)
\end{align*}
for all $j \in I_{t-1}$, we have
\[l_{21} - l_{11} = l_{31} - l_{21} = \cdots =  l_{t1} - l_{(t-1)1} = l_{11} - l_{t1} =: \iota.\]
As $t\iota = (l_{21}-l_{11})+(l_{31}-l_{21})+ \cdots + (l_{t1}-l_{(t-1)1})+(l_{11}-l_{t1}) = 0,$
we know that $\iota=0$ and thus $l_{21} = l_{11}$. This implies that $t=1$ and $\sum_{i=1}^s \mathrm{lv}(\epsilon_i, \epsilon_{i+1})=0$, where $\epsilon_{s+1} := \epsilon_1$.

To show that $\sum_{i=1}^s \epsilon_i = 0$, note that for every $x \in \mathbb{R}$, we have 
\begin{equation} \label{equation:basicmaxminrule}
\textrm{max}\{x,0\}+\textrm{min}\{x,0\}=x.
\end{equation}
We have
\begin{align*}
\sum_{i=1}^s \epsilon_i &= \sum_{i=1}^s (\textrm{max}\{\epsilon_i,0\} + \textrm{min}\{\epsilon_i,0\}) & \text{as } \eqref{equation:basicmaxminrule}\\
&= \sum_{i=1}^s (\textrm{max}\{\epsilon_i,0\} + \textrm{min}\{\epsilon_{i+1},0\}) & \text{as } \epsilon_{s+1}:=\epsilon_1 \\
&= \sum_{i=1}^s \mathrm{lv}(\epsilon_i, \epsilon_{i+1})=0. & \text{as } \eqref{equation:changeofweight} & \qedhere\\
\end{align*}
\end{proof}

\begin{corollary} \label{corollary:connectedcomponentfirstcalc}
Let $\epsilon$ be a circular sequence of $s$ integers (thus $|\epsilon|=s$). We have $w(\Gamma^m_{\epsilon}) = c(\Gamma^m_{\epsilon}) \times s$.
\end{corollary}
\begin{proof}
By Proposition \ref{proposition:circularzero}, each circular digraph has $s$ vertices and thus $s$ edges. Hence $w(\Gamma^m_{\epsilon}) = c(\Gamma^m_{\epsilon}) \times s$.
\end{proof}

\subsection{First Reduction} \label{subs:firstreduction}

\begin{theorem}[First Reduction] \label{firstreduction}
Let $\epsilon = (\epsilon_1, \epsilon_2, \dots, \epsilon_s)$ be a circular sequence of $s \geq 2$ integers. Let $\epsilon' = (\epsilon_1+\epsilon_2, \epsilon_3, \dots, \epsilon_s)$. If $\epsilon_1 \epsilon_{2} > 0$, then $c(\Gamma^m_{\epsilon}) = c(\Gamma^m_{\epsilon'})$ and $\ell(\Gamma^m_{\epsilon}) = \ell(\Gamma^m_{\epsilon'})$.
\end{theorem}

We will prove Theorem \ref{firstreduction} in a series of lemmas. The next lemma proves Theorem \ref{firstreduction} when $\epsilon = (\epsilon_1, \epsilon_2)$ is a circular sequence of two integers.

\begin{lemma} \label{lemma:firstredcase2}
Let $\epsilon = (\epsilon_1, \epsilon_2)$ be a circular sequence of two integers with $\epsilon_1\epsilon_2 > 0$. Let $\epsilon' = (\epsilon_1+\epsilon_2)$. We have $c(\Gamma^m_{\epsilon}) = c(\Gamma^m_{\epsilon'})=0$ and $\ell(\Gamma^m_{\epsilon}) = \ell(\Gamma^m_{\epsilon'})=0$. 
\end{lemma}
\begin{proof}
If $\epsilon_1, \epsilon_2 > 0$ (resp. $\epsilon_1, \epsilon_2 < 0$), then from Cases (10) and (11) (resp. Cases (8) and (9)) in Subsection \ref{subsection:weighteddigraphs}, we know that all vertices of $\Gamma^m_{\epsilon}$ are regarded as zero. As $\epsilon_1+\epsilon_2 \neq 0$, by Remark \ref{remark:digraph1}, we know that all vertices of $\Gamma^m_{\epsilon'}$ are regarded as zero. Thus we have $c(\Gamma^m_{\epsilon}) = c(\Gamma^m_{\epsilon'})=0$ and $\ell(\Gamma^m_{\epsilon}) = \ell(\Gamma^m_{\epsilon'})=0$.
\end{proof}

Next we study the case when the circular sequence $\epsilon$ has more than two integers.

\begin{lemma} \label{lemma:firstredbasiclemma}
To prove Theorem \ref{firstreduction}, it is enough to assume $\epsilon_1, \epsilon_2 > 0$.
\end{lemma}
\begin{proof}
Suppose $\epsilon_1, \epsilon_2 < 0$. Let 
\begin{equation*}
\begin{aligned}
\eta    & =   (\eta_s, \dots, \eta_3, \eta_2, \eta_1) & =  (-\epsilon_s, \dots, -\epsilon_3, -\epsilon_2, -\epsilon_1), \\
\eta'   & =  (\eta'_s, \dots, \eta'_3, \eta'_1)   & =  (-\epsilon_s, \dots, -\epsilon_3, -\epsilon_2-\epsilon_1).
\end{aligned}
\end{equation*}
By Remark \ref{remark:negativereverse}, we know that $\Gamma^m_{\eta}$ (resp. $\Gamma^m_{\eta'}$) can be obtained from $\Gamma^m_{\epsilon}$ (resp. $\Gamma^m_{\epsilon'}$) by reversing all the edges and taking additive inverses of all the weights. As a result, we have 
\[c(\Gamma^m_{\epsilon}) = c(\Gamma^m_{\eta}), \; \ell(\Gamma^m_{\epsilon}) = \ell(\Gamma^m_{\eta}), \; c(\Gamma^m_{\epsilon'}) = c(\Gamma^m_{\eta'}), \; \ell(\Gamma^m_{\epsilon'}) = \ell(\Gamma^m_{\eta'}).\]
If Theorem \ref{firstreduction} is true when $\epsilon_1, \epsilon_2 > 0$, then 
\[c(\Gamma^m_{\eta}) = c(\Gamma^m_{\eta'}), \qquad \ell(\Gamma^m_{\eta}) = \ell(\Gamma^m_{\eta'})\]
as $\eta_1, \eta_2 > 0$. Hence
\[c(\Gamma^m_{\epsilon}) = c(\Gamma^m_{\eta}) = c(\Gamma^m_{\eta'}) = c(\Gamma^m_{\epsilon'}), \qquad \ell(\Gamma^m_{\epsilon}) = \ell(\Gamma^m_{\eta}) = \ell(\Gamma^m_{\eta'}) = \ell(\Gamma^m_{\epsilon'}).\qedhere\]
\end{proof}

For the remaining of this subsection, we assume that $\epsilon_1, \epsilon_{2} > 0$. Let 
\[\Gamma^m := \Gamma^m_{\epsilon_1, \epsilon_2} \cup \Gamma^m_{\epsilon_2, \epsilon_3}, \qquad {\Gamma'}^m := \Gamma^m_{\epsilon_1+\epsilon_2, \epsilon_3}.\]
For consistency, we use $\{x_{0,1}, \dots, x_{m-1,1}, x_{0,3}, \dots, x_{m-1,3}\}$ as the vertex set of ${\Gamma'}^m$.  

\begin{lemma} \label{lemma:firstredlemma1}
Let $\epsilon = (\epsilon_1, \epsilon_2, \dots, \epsilon_s)$ be a circular sequence of at least three integers. Let $\epsilon' = (\epsilon_1+\epsilon_2, \epsilon_3, \dots, \epsilon_s)$. If $\epsilon_1, \epsilon_{2} > 0$ and $\epsilon_1+\epsilon_2 < m$, then $c(\Gamma^m_{\epsilon}) = c(\Gamma^m_{\epsilon'})$ and $\ell(\Gamma^m_{\epsilon}) = \ell(\Gamma^m_{\epsilon'})$.
\end{lemma}
\begin{proof}
We construct $\Gamma^m$ and ${\Gamma'}^m$ in six mutually exclusive cases. 

\begin{case}
\item Suppose that $\epsilon_3 \geq 0$.

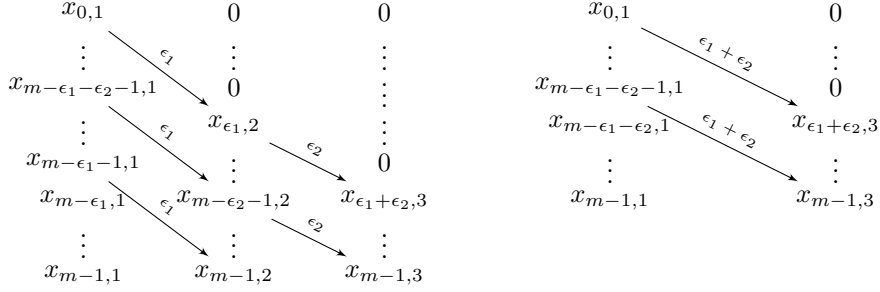
\begin{figure}[H]
\centering
\begin{tikzpicture}
\tikzset{vertex/.style = {shape=circle,draw,minimum size=1.5em}}
\tikzset{edge/.style = {->,> = latex'}}
\node (a) at (0,0) {$x_{0,1}$};
\node (b) at (0,-0.5) {$\vdots$};
\node (c) at (0,-1) {$x_{m-\epsilon_1-\epsilon_2-1,1}$};
\node (d) at (0,-1.5) {$\vdots$};
\node (e) at (0,-2) {$x_{m-\epsilon_1-1,1}$};
\node (f) at (0,-2.5) {$x_{m-\epsilon_1,1}$};
\node (g) at (0,-3) {$\vdots$};
\node (h) at (0,-3.5) {$x_{m-1,1}$};
\node (i) at (2,0) {$0$};
\node (j) at (2,-0.5) {$\vdots$};
\node (k) at (2,-1) {$0$};
\node (l) at (2,-1.5) {$x_{\epsilon_1,2}$};
\node (m) at (2,-2) {$\vdots$};
\node (n) at (2,-2.5) {$x_{m-\epsilon_2-1,2}$};
\node (o) at (2,-3) {$\vdots$};
\node (p) at (2,-3.5) {$x_{m-1,2}$};
\node (q) at (4,0) {$0$};
\node (r) at (4,-0.5) {$\vdots$};
\node (s) at (4,-1) {$\vdots$};
\node (t) at (4,-1.5) {$\vdots$};
\node (u) at (4,-2)  {$0$};
\node (v) at (4,-2.5) {$x_{\epsilon_1+\epsilon_2,3}$};
\node (w) at (4,-3) {$\vdots$};
\node (x) at (4,-3.5) {$x_{m-1,3}$};
\draw[edge] (a) -- (l) node[sloped,midway,anchor=center, above] {\tiny $\epsilon_{1}$};
\draw[edge] (c) -- (n) node[sloped,midway,anchor=center, above] {\tiny $\epsilon_{1}$};
\draw[edge] (e) -- (p) node[sloped,midway,anchor=center, above] {\tiny $\epsilon_{1}$};
\draw[edge] (l) -- (v) node[sloped,midway,anchor=center, above] {\tiny $\epsilon_{2}$};
\draw[edge] (n) -- (x) node[sloped,midway,anchor=center, above] {\tiny $\epsilon_{2}$};
\node (aa) at (7,0) {$x_{0,1}$};
\node (bb) at (7,-0.5) {$\vdots$};
\node (cc) at (7,-1) {$x_{m-\epsilon_1-\epsilon_2-1,1}$};
\node (dd) at (7,-1.5) {$x_{m-\epsilon_1-\epsilon_2,1}$};
\node (ee) at (7,-2) {$\vdots$};
\node (hh) at (7,-2.5) {$x_{m-1,1}$};
\node (ii) at (10,0) {$0$};
\node (jj) at (10,-0.5) {$\vdots$};
\node (kk) at (10,-1) {$0$};
\node (ll) at (10,-1.5) {$x_{\epsilon_1+\epsilon_2,3}$};
\node (mm) at (10,-2) {$\vdots$};
\node (pp) at (10,-2.5) {$x_{m-1,3}$};
\draw[edge] (aa) -- (ll) node[sloped,midway,anchor=center, above] {\tiny $\epsilon_{1}+\epsilon_2$};
\draw[edge] (cc) -- (pp) node[sloped,midway,anchor=center, above] {\tiny $\epsilon_{1}+\epsilon_2$};
\end{tikzpicture}
\caption{Digraphs for $\Gamma^m$ and ${\Gamma'}^m$ in Case (1)} \label{fig:case12}
\end{figure}

\item Suppose that $\epsilon_3 < 0$ and $\epsilon_3 \in \{-1, -2, \dots, -\epsilon_2+1\}$.

\begin{figure}[H]
\centering
\begin{tikzpicture}
\tikzset{vertex/.style = {shape=circle,draw,minimum size=1.5em}}
\tikzset{edge/.style = {->,> = latex'}}
\node (a) at (0,0) {$x_{0,1}$};
\node (b) at (0,-0.5) {$\vdots$};
\node (c) at (0,-1) {$x_{m-\epsilon_1-\epsilon_2-1,1}$};
\node (d) at (0,-1.5) {$\vdots$};
\node (e) at (0,-2) {$\vdots$};
\node (f) at (0,-2.5) {$x_{m-\epsilon_1-1,1}$};
\node (g) at (0,-3) {$\vdots$};
\node at (0,-3.5) {$\vdots$};
\node (h) at (0,-4) {$x_{m-1,1}$};
\node (i) at (2,0) {$0$};
\node (j) at (2,-0.5) {$\vdots$};
\node (k) at (2,-1) {$0$};
\node (l) at (2,-1.5) {$x_{\epsilon_1,2}$};
\node (m) at (2,-2) {$\vdots$};
\node (n) at (2,-2.5) {$x_{m-\epsilon_2-1,2}$};
\node (o) at (2,-3) {$\vdots$};
\node at (2,-3.5) {$\vdots$};
\node (p) at (2,-4) {$x_{m-1,2}$};
\node (q) at (4.5,0) {$0$};
\node (r) at (4.5,-0.5) {$\vdots$};
\node (s) at (4.5,-1) {$\vdots$};
\node (t) at (4.5,-1.5) {$0$};
\node (u) at (4.5,-2)  {$x_{\epsilon_1+\epsilon_2+\epsilon_3,3}$};
\node (v) at (4.5,-2.5) {$\vdots$};
\node (w) at (4.5,-3) {$x_{m+\epsilon_3-1,3}$};
\node (x) at (4.5,-3.5) {$\vdots$};
\node (y) at (4.5,-4) {$x_{m-1,3}$};
\draw[edge] (a) -- (l) node[sloped,midway,anchor=center, above] {\tiny $\epsilon_{1}$};
\draw[edge] (c) -- (n) node[sloped,midway,anchor=center, above] {\tiny $\epsilon_{1}$};
\draw[edge] (f) -- (p) node[sloped,midway,anchor=center, above] {\tiny $\epsilon_{1}$};
\draw[edge] (l) -- (u) node[sloped,midway,anchor=center, above] {\tiny $\epsilon_{2}+\epsilon_3$};
\draw[edge] (n) -- (w) node[sloped,midway,anchor=center, above] {\tiny $\epsilon_{2}+\epsilon_3$};
\node (aa) at (7,0) {$x_{0,1}$};
\node (bb) at (7,-0.5) {$\vdots$};
\node (cc) at (7,-1) {$x_{m-\epsilon_1-\epsilon_2-1,1}$};
\node (dd) at (7,-1.5) {$x_{m-\epsilon_1-\epsilon_2,1}$};
\node (ee) at (7,-2) {$\vdots$};
\node (ff) at (7,-2.5) {$\vdots$};
\node (gg) at (7,-3)  {$\vdots$};
\node (hh) at (7,-3.5) {$x_{m-1,1}$};
\node (ii) at (10,0) {$0$};
\node (jj) at (10,-0.5) {$\vdots$};
\node (kk) at (10,-1) {$0$};
\node (ll) at (10,-1.5) {$x_{\epsilon_1+\epsilon_2+\epsilon_3,3}$};
\node (mm) at (10,-2) {$\vdots$};
\node (oo) at (10,-2.5) {$x_{m+\epsilon_3-1,3}$};
\node (pp) at (10,-3) {$\vdots$};
\node (qq) at (10,-3.5) {$x_{m-1,3}$};
\draw[edge] (aa) -- (ll) node[sloped,midway,anchor=center, above] {\tiny $\epsilon_{1}+\epsilon_2+\epsilon_3$};
\draw[edge] (cc) -- (oo) node[sloped,midway,anchor=center, above] {\tiny $\epsilon_{1}+\epsilon_2+\epsilon_3$};
\end{tikzpicture}
\caption{Digraphs for $\Gamma^m$ and ${\Gamma'}^m$ in Case (2)} \label{fig:case14}
\end{figure}
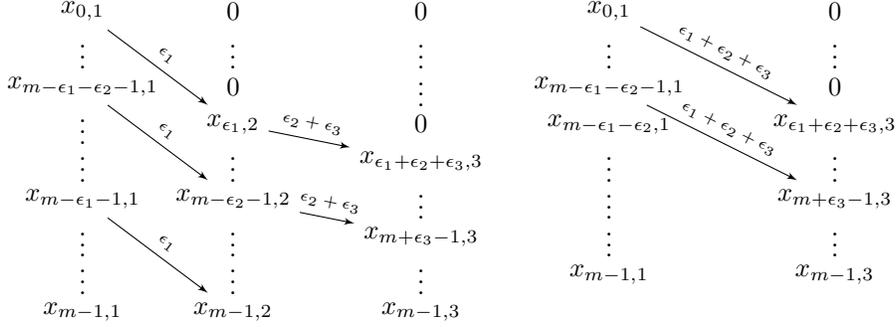

\item Suppose that $\epsilon_3 < 0$ and $\epsilon_3 = -\epsilon_2$.

\begin{figure}[H]
\centering
\begin{tikzpicture}
\tikzset{vertex/.style = {shape=circle,draw,minimum size=1.5em}}
\tikzset{edge/.style = {->,> = latex'}}
\node (a) at (0,0) {$x_{0,1}$};
\node (b) at (0,-0.5) {$\vdots$};
\node (c) at (0,-1) {$x_{m-\epsilon_1-\epsilon_2-1,1}$};
\node (d) at (0,-1.5) {$\vdots$};
\node (e) at (0,-2) {$x_{m-\epsilon_1-1,1}$};
\node (f) at (0,-2.5) {$x_{m-\epsilon_1,1}$};
\node (g) at (0,-3) {$\vdots$};
\node (h) at (0,-3.5) {$x_{m-1,1}$};
\node (i) at (2,0) {$0$};
\node (j) at (2,-0.5) {$\vdots$};
\node (k) at (2,-1) {$0$};
\node (l) at (2,-1.5) {$x_{\epsilon_1,2}$};
\node (m) at (2,-2) {$\vdots$};
\node (n) at (2,-2.5) {$x_{m-\epsilon_2-1,2}$};
\node (o) at (2,-3) {$\vdots$};
\node (p) at (2,-3.5) {$x_{m-1,2}$};
\node (q) at (4,0) {$0$};
\node (r) at (4,-0.5) {$\vdots$};
\node (s) at (4,-1) {$0$};
\node (t) at (4,-1.5) {$x_{\epsilon_1,3}$};
\node (u) at (4,-2)  {$\vdots$};
\node (v) at (4,-2.5) {$x_{m-\epsilon_2-1,3}$};
\node (w) at (4,-3) {$\vdots$};
\node (x) at (4,-3.5) {$x_{m-1,3}$};
\draw[edge] (a) -- (l) node[sloped,midway,anchor=center, above] {\tiny $\epsilon_{1}$};
\draw[edge] (c) -- (n) node[sloped,midway,anchor=center, above] {\tiny $\epsilon_{1}$};
\draw[edge] (e) -- (p) node[sloped,midway,anchor=center, above] {\tiny $\epsilon_{1}$};
\draw[edge] (l) -- (t) node[sloped,midway,anchor=center, above] {\tiny $0$};
\draw[edge] (n) -- (v) node[sloped,midway,anchor=center, above] {\tiny $0$};
\node (aa) at (7,0) {$x_{0,1}$};
\node (bb) at (7,-0.5) {$\vdots$};
\node (cc) at (7,-1) {$x_{m-\epsilon_1-\epsilon_2-1,1}$};
\node (dd) at (7,-1.5) {$x_{m-\epsilon_1-\epsilon_2,1}$};
\node (ee) at (7,-2) {$\vdots$};
\node (ff) at (7,-2.5) {$\vdots$};
\node (gg) at (7,-3)  {$\vdots$};
\node (hh) at (7,-3.5) {$x_{m-1,1}$};
\node (ii) at (10,0) {$0$};
\node (jj) at (10,-0.5) {$\vdots$};
\node (kk) at (10,-1) {$0$};
\node (ll) at (10,-1.5) {$x_{\epsilon_1,3}$};
\node (mm) at (10,-2) {$\vdots$};
\node (oo) at (10,-2.5) {$x_{m-\epsilon_2-1,3}$};
\node (pp) at (10,-3) {$\vdots$};
\node (qq) at (10,-3.5) {$x_{m-1,3}$};
\draw[edge] (aa) -- (ll) node[sloped,midway,anchor=center, above] {\tiny $\epsilon_{1}$};
\draw[edge] (cc) -- (oo) node[sloped,midway,anchor=center, above] {\tiny $\epsilon_{1}$};
\end{tikzpicture}
\caption{Digraphs for $\Gamma^m$ and ${\Gamma'}^m$ for Case (3)} \label{fig:case13}
\end{figure}
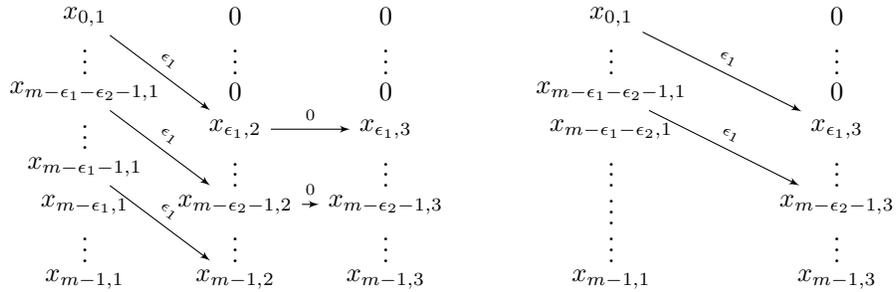

\item Suppose that $\epsilon_3 < 0$ and $\epsilon_3 \in \{-\epsilon_2-1, -\epsilon_2-2, \dots, -\epsilon_2-\epsilon_1\}$.

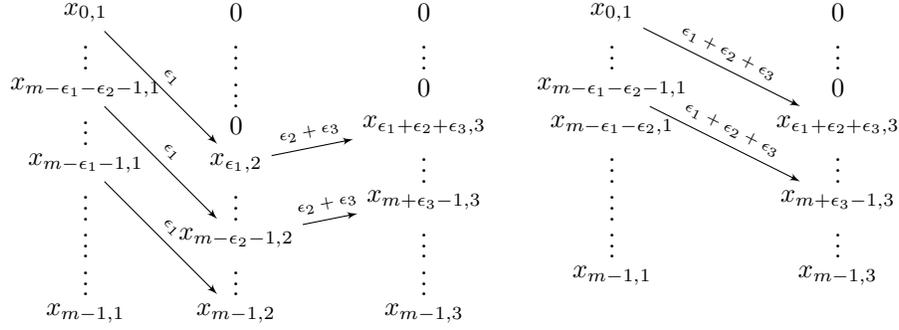
\begin{figure}[H]
\centering
\begin{tikzpicture}
\tikzset{vertex/.style = {shape=circle,draw,minimum size=1.5em}}
\tikzset{edge/.style = {->,> = latex'}}
\node (a) at (0,0) {$x_{0,1}$};
\node (b) at (0,-0.5) {$\vdots$};
\node (c) at (0,-1) {$x_{m-\epsilon_1-\epsilon_2-1,1}$};
\node (d) at (0,-1.5) {$\vdots$};
\node (e) at (0,-2) {$x_{m-\epsilon_1-1,1}$};
\node (f) at (0,-2.5) {$\vdots$};
\node (g) at (0,-3) {$\vdots$};
\node at (0,-3.5) {$\vdots$};
\node (h) at (0,-4) {$x_{m-1,1}$};
\node (i) at (2,0) {$0$};
\node (j) at (2,-0.5) {$\vdots$};
\node (k) at (2,-1) {$\vdots$};
\node (l) at (2,-1.5) {$0$};
\node (m) at (2,-2) {$x_{\epsilon_1,2}$};
\node (n) at (2,-2.5) {$\vdots$};
\node (o) at (2,-3) {$x_{m-\epsilon_2-1,2}$};
\node at (2,-3.5) {$\vdots$};
\node (p) at (2,-4) {$x_{m-1,2}$};
\node (q) at (4.5,0) {$0$};
\node (r) at (4.5,-0.5) {$\vdots$};
\node (s) at (4.5,-1) {$0$};
\node (t) at (4.5,-1.5) {$x_{\epsilon_1+\epsilon_2+\epsilon_3,3}$};
\node (u) at (4.5,-2)  {$\vdots$};
\node (v) at (4.5,-2.5) {$x_{m+\epsilon_3-1,3}$};
\node (w) at (4.5,-3) {$\vdots$};
\node (x) at (4.5,-3.5) {$\vdots$};
\node (y) at (4.5,-4) {$x_{m-1,3}$};
\draw[edge] (a) -- (m) node[sloped,midway,anchor=center, above] {\tiny $\epsilon_{1}$};
\draw[edge] (c) -- (o) node[sloped,midway,anchor=center, above] {\tiny $\epsilon_{1}$};
\draw[edge] (e) -- (p) node[sloped,midway,anchor=center, above] {\tiny $\epsilon_{1}$};
\draw[edge] (m) -- (t) node[sloped,midway,anchor=center, above] {\tiny $\epsilon_{2}+\epsilon_3$};
\draw[edge] (o) -- (v) node[sloped,midway,anchor=center, above] {\tiny $\epsilon_{2}+\epsilon_3$};
\node (aa) at (7,0) {$x_{0,1}$};
\node (bb) at (7,-0.5) {$\vdots$};
\node (cc) at (7,-1) {$x_{m-\epsilon_1-\epsilon_2-1,1}$};
\node (dd) at (7,-1.5) {$x_{m-\epsilon_1-\epsilon_2,1}$};
\node (ee) at (7,-2) {$\vdots$};
\node (ff) at (7,-2.5) {$\vdots$};
\node (gg) at (7,-3)  {$\vdots$};
\node (hh) at (7,-3.5) {$x_{m-1,1}$};
\node (ii) at (10,0) {$0$};
\node (jj) at (10,-0.5) {$\vdots$};
\node (kk) at (10,-1) {$0$};
\node (ll) at (10,-1.5) {$x_{\epsilon_1+\epsilon_2+\epsilon_3,3}$};
\node (mm) at (10,-2) {$\vdots$};
\node (oo) at (10,-2.5) {$x_{m+\epsilon_3-1,3}$};
\node (pp) at (10,-3) {$\vdots$};
\node (qq) at (10,-3.5) {$x_{m-1,3}$};
\draw[edge] (aa) -- (ll) node[sloped,midway,anchor=center, above] {\tiny $\epsilon_{1}+\epsilon_2+\epsilon_3$};
\draw[edge] (cc) -- (oo) node[sloped,midway,anchor=center, above] {\tiny $\epsilon_{1}+\epsilon_2+\epsilon_3$};
\end{tikzpicture}
\caption{Digraphs for $\Gamma^m$ and ${\Gamma'}^m$ in Case (4)} \label{fig:case15}
\end{figure}

\item Suppose that $\epsilon_3 < 0$ and $\epsilon_3 \in \{-\epsilon_2-\epsilon_1-1, -\epsilon_2-\epsilon_1-2, \dots, -m+1\}$.

\begin{figure}[H]
\centering
\begin{tikzpicture}
\tikzset{vertex/.style = {shape=circle,draw,minimum size=1.5em}}
\tikzset{edge/.style = {->,> = latex'}}
\node (a) at (0,0) {$0$};
\node (b) at (0,-0.5) {$\vdots$};
\node (c) at (0,-1) {$0$};
\node (d) at (0,-1.5) {$x_{-\epsilon_1-\epsilon_2-\epsilon_3,1}$};
\node (e) at (0,-2) {$\vdots$};
\node (f) at (0,-2.5) {$x_{m-\epsilon_1-\epsilon_2-1,1}$};
\node (g) at (0,-3) {$\vdots$};
\node (h) at (0,-3.5) {$x_{m-\epsilon_1-1,1}$};
\node  at (0,-4) {$\vdots$};
\node at (0,-4.5) {$x_{m-1,1}$};
\node (i) at (2,0) {$0$};
\node (j) at (2,-0.5) {$\vdots$};
\node (k) at (2,-1) {$\vdots$};
\node (l) at (2,-1.5) {$\vdots$};
\node at (2,-2) {$0$};
\node (m) at (2,-2.5) {$x_{-\epsilon_2-\epsilon_3,2}$};
\node (n) at (2,-3) {$\vdots$};
\node (o) at (2,-3.5) {$x_{m-\epsilon_2-1,2}$};
\node at (2,-4) {$\vdots$};
\node (p) at (2,-4.5) {$x_{m-1,2}$};
\node (q) at (4.5,0) {$x_{0,3}$};
\node (r) at (4.5,-0.5) {$\vdots$};
\node (s) at (4.5,-1) {$x_{m+\epsilon_3-1,3}$};
\node (t) at (4.5,-1.5) {$\vdots$};
\node (u) at (4.5,-2)  {$\vdots$};
\node (v) at (4.5,-2.5) {$\vdots$};
\node (w) at (4.5,-3) {$\vdots$};
\node (x) at (4.5,-3.5) {$\vdots$};
\node at (4.5,-4) {$\vdots$};
\node (y) at (4.5,-4.5) {$x_{m-1,3}$};
\draw[edge] (d) -- (m) node[sloped,midway,anchor=center, above] {\tiny $\epsilon_{1}$};
\draw[edge] (f) -- (o) node[sloped,midway,anchor=center, above] {\tiny $\epsilon_{1}$};
\draw[edge] (h) -- (p) node[sloped,midway,anchor=center, above] {\tiny $\epsilon_{1}$};
\draw[edge] (m) -- (q) node[sloped,midway,anchor=center, above] {\tiny $\epsilon_{2}+\epsilon_3$};
\draw[edge] (o) -- (s) node[sloped,midway,anchor=center, above] {\tiny $\epsilon_{2}+\epsilon_3$};
\node (aa) at (7,0) {$0$};
\node (bb) at (7,-0.5) {$\vdots$};
\node (cc) at (7,-1) {$0$};
\node (dd) at (7,-1.5) {$x_{-\epsilon_1-\epsilon_2-\epsilon_3,1}$};
\node (ee) at (7,-2) {$\vdots$};
\node (ff) at (7,-2.5) {$x_{m-\epsilon_1-\epsilon_2-1,1}$};
\node (gg) at (7,-3)  {$\vdots$};
\node (hh) at (7,-3.5) {$x_{m-1,1}$};
\node (ii) at (10,0) {$x_{0,3}$};
\node (jj) at (10,-0.5) {$\vdots$};
\node (kk) at (10,-1) {$x_{m+\epsilon_3-1,3}$};
\node (ll) at (10,-1.5) {$\vdots$};
\node (mm) at (10,-2) {$\vdots$};
\node (oo) at (10,-2.5) {$\vdots$};
\node (pp) at (10,-3) {$\vdots$};
\node (qq) at (10,-3.5) {$x_{m-1,3}$};
\draw[edge] (dd) -- (ii) node[sloped,midway,anchor=center, above] {\tiny $\epsilon_{1}+\epsilon_2+\epsilon_3$};
\draw[edge] (ff) -- (kk) node[sloped,midway,anchor=center, above] {\tiny $\epsilon_{1}+\epsilon_2+\epsilon_3$};
\end{tikzpicture}
\caption{Digraphs for $\Gamma^m$ and ${\Gamma'}^m$ in Case (5)} \label{fig:case16}
\end{figure}
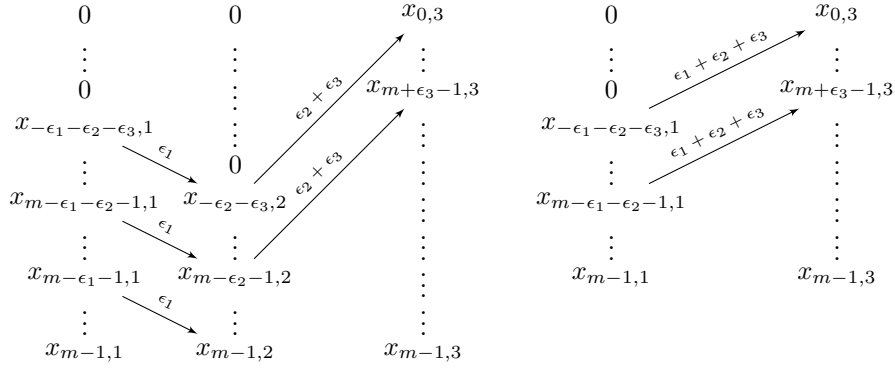

\item Suppose that $\epsilon_3 < 0$ and $-\epsilon_3 \in \{-m,-m-1,\dots\}$.
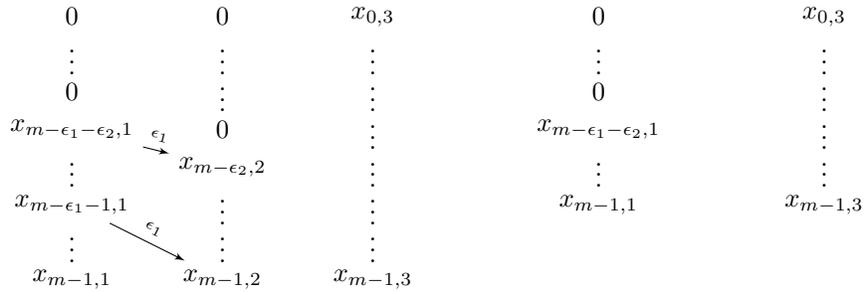
\begin{figure}[H]
\centering
\begin{tikzpicture}
\tikzset{vertex/.style = {shape=circle,draw,minimum size=1.5em}}
\tikzset{edge/.style = {->,> = latex'}}
\node (a) at (0,0) {$0$};
\node (b) at (0,-0.5) {$\vdots$};
\node (c) at (0,-1) {$0$};
\node (d) at (0,-1.5) {$x_{m-\epsilon_1-\epsilon_2,1}$};
\node (f) at (0,-2) {\vdots};
\node (g) at (0,-2.5) {$x_{m-\epsilon_1-1,1}$};
\node (h) at (0,-3) {$\vdots$};
\node at (0,-3.5) {$x_{m-1,1}$};
\node (i) at (2,0) {$0$};
\node (k) at (2,-0.5) {$\vdots$};
\node (l) at (2,-1) {$\vdots$};
\node  at (2,-1.5) {$0$};
\node (m) at (2,-2) {$x_{m-\epsilon_2,2}$};
\node (o) at (2,-2.5) {$\vdots$};
\node at (2,-3) {$\vdots$};
\node (p) at (2,-3.5) {$x_{m-1,2}$};
\node (q) at (4,0) {$x_{0,3}$};
\node (s) at (4,-0.5) {$\vdots$};
\node (t) at (4,-1) {$\vdots$};
\node (u) at (4,-1.5)  {$\vdots$};
\node (v) at (4,-2) {$\vdots$};
\node (w) at (4,-2.5) {$\vdots$};
\node (x) at (4,-3) {$\vdots$};
\node at (4,-3.5) {$x_{m-1,3}$};
\draw[edge] (d) -- (m) node[sloped,midway,anchor=center, above] {\tiny $\epsilon_{1}$};
\draw[edge] (g) -- (p) node[sloped,midway,anchor=center, above] {\tiny $\epsilon_{1}$};
\node (aa) at (7,0) {$0$};
\node (bb) at (7,-0.5) {$\vdots$};
\node (cc) at (7,-1) {$0$};
\node (dd) at (7,-1.5) {$x_{m-\epsilon_1-\epsilon_2,1}$};
\node (ee) at (7,-2) {$\vdots$};
\node (ff) at (7,-2.5) {$x_{m-1,1}$};
\node (ii) at (10,0) {$x_{0,3}$};
\node (jj) at (10,-0.5) {$\vdots$};
\node (kk) at (10,-1) {$\vdots$};
\node (ll) at (10,-1.5) {$\vdots$};
\node (mm) at (10,-2) {$\vdots$};
\node (oo) at (10,-2.5) {$x_{m-1,3}$};
\end{tikzpicture}
\caption{Digraphs for $\Gamma^m$ and ${\Gamma'}^m$ in Case (6)} \label{fig:case17}
\end{figure}
\end{case}
In each one of the six cases it is clear that if we replace $\Gamma^m$ with ${\Gamma'}^m$ in $\Gamma^m_{\epsilon}$, the number of free linear digraphs and the number of circular digraphs will not change. 
\end{proof}

\begin{lemma} \label{lemma:firstredlemma2}
Let $\epsilon = (\epsilon_1, \epsilon_2, \dots, \epsilon_s)$ be a circular sequence of at least three integers. Let $\epsilon' = (\epsilon_1+\epsilon_2, \epsilon_3, \dots, \epsilon_s)$. If $\epsilon_1, \epsilon_{2} > 0$ and $\epsilon_1+\epsilon_2 \geq m$, then $c(\Gamma^m_{\epsilon}) = c(\Gamma^m_{\epsilon'})$ and $\ell(\Gamma^m_{\epsilon}) = \ell(\Gamma^m_{\epsilon'})$.
\end{lemma}
\begin{proof}
The proof of Lemma \ref{lemma:firstredlemma2} is entirely similar to the proof of Lemma \ref{lemma:firstredlemma1} by considering cases. As $\epsilon_1+\epsilon_2 \geq m$, the digraph ${\Gamma'}^m$ has no edges and is completely determined by specifying its zero vertices. If $\epsilon_3 \leq -m$, then ${\Gamma'}^m$ has no zero vertices. If $\epsilon_3 \geq 0$, then the zero vertices of ${\Gamma'}^m$ are precisely the vertices $x_{0,3}, \dots, x_{m-1,3}$. If $-m < \epsilon_3 < 0$, then ${\Gamma'}^m$ is as in Figure \ref{fig:case18} below. In Figure \ref{fig:case18} we present the digraph $\Gamma^m$ in only one of the possible cases and we omit it for all other cases.

\begin{figure}[H]
\centering
\begin{tikzpicture}
\tikzset{vertex/.style = {shape=circle,draw,minimum size=1.5em}}
\tikzset{edge/.style = {->,> = latex'}}
\node (a) at (0,0) {$x_{0,1}$};
\node (b) at (0,-0.5) {$\vdots$};
\node (c) at (0,-1) {$x_{m-\epsilon_1-1,1}$};
\node (d) at (0,-1.5) {$\vdots$};
\node (e) at (0,-2) {$\vdots$};
\node (f) at (0,-2.5) {$\vdots$};
\node (g) at (0,-3) {$\vdots$};
\node at (0,-3.5) {$\vdots$};
\node (h) at (0,-4) {$x_{m-1,1}$};
\node (i) at (3.5,0) {$0$};
\node (j) at (3.5,-0.5) {$\vdots$};
\node (k) at (3.5,-1) {$x_{-\epsilon_3+\epsilon_2,2}=0$};
\node (l) at (3.5,-1.5) {$\vdots$};
\node (m) at (3.5,-2) {$x_{m-\epsilon_2-1,2}=0$};
\node (n) at (3.5,-2.5) {$\vdots$};
\node (o) at (3.5,-3) {$x_{\epsilon_1,2}$};
\node at (3.5,-3.5) {$\vdots$};
\node (p) at (3.5,-4) {$x_{m-1,2}$};
\node (q) at (6,0) {$x_{0,3}=0$};
\node (r) at (6,-0.5) {$\vdots$};
\node (s) at (6,-1) {$x_{m+\epsilon_3-1,3}=0$};
\node (t) at (6,-1.5) {$x_{m+\epsilon_3,3}$};
\node (u) at (6,-2)  {$\vdots$};
\node (v) at (6,-2.5) {$\vdots$};
\node (w) at (6,-3) {$\vdots$};
\node (x) at (6,-3.5) {$\vdots$};
\node (y) at (6,-4) {$x_{m-1,3}$};
\draw[edge] (a) -- (o) node[sloped,midway,anchor=center, above] {\tiny $\epsilon_{1}$};
\draw[edge] (c) -- (p) node[sloped,midway,anchor=center, above] {\tiny $\epsilon_{1}$};
\draw[edge] (k) -- (q) node[sloped,midway,anchor=center, above] {\tiny $\epsilon_{2}+\epsilon_3$};
\draw[edge] (m) -- (s) node[sloped,midway,anchor=center, above] {\tiny $\epsilon_{2}+\epsilon_3$};
\node (aa) at (8,0) {$x_{0,1}$};
\node (bb) at (8,-0.5) {$\vdots$};
\node (cc) at (8,-1) {$\vdots$};
\node (dd) at (8,-1.5) {$\vdots$};
\node (ee) at (8,-2) {$\vdots$};
\node (ff) at (8,-2.5) {$\vdots$};
\node (gg) at (8,-3)  {$\vdots$};
\node (hh) at (8,-3.5) {$x_{m-1,1}$};
\node (ii) at (10,0) {$0$};
\node (jj) at (10,-0.5) {$\vdots$};
\node (kk) at (10,-1) {$0$};
\node (ll) at (10,-1.5) {$x_{m+\epsilon_3,3}$};
\node (mm) at (10,-2) {$\vdots$};
\node (oo) at (10,-2.5) {$\vdots$};
\node (pp) at (10,-3) {$\vdots$};
\node (qq) at (10,-3.5) {$x_{m-1,3}$}; 
\end{tikzpicture}
\caption{Digraphs for $\Gamma^m$ and ${\Gamma'}^m$ when $0 \leq \epsilon_1, \epsilon_2 < m$, $\epsilon_3 <0$ with $\epsilon_1+\epsilon_2 \geq m > -\epsilon_3 > \epsilon_2$} \label{fig:case18} \qedhere
\end{figure}
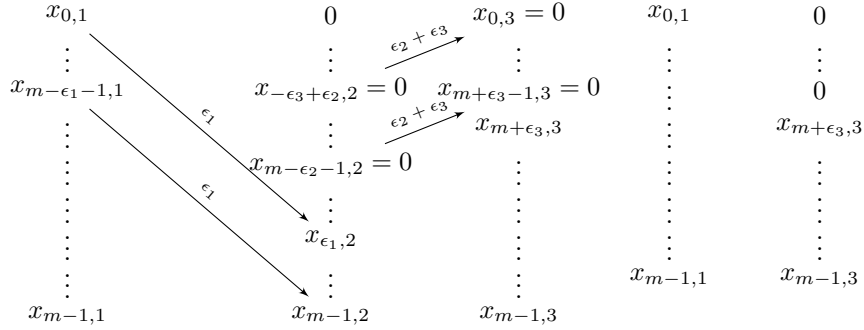
\end{proof}

\begin{proof}[Proof of Theorem \ref{firstreduction}]
Lemma \ref{lemma:firstredbasiclemma} allows us to only consider the case when both $\epsilon_1$ and $\epsilon_2$ are positive. Lemma \ref{lemma:firstredcase2} proves Theorem \ref{firstreduction} when the sequence $\epsilon$ has only two integers. Lemmas \ref{lemma:firstredlemma1} and \ref{lemma:firstredlemma2} prove Theorem \ref{firstreduction} when the sequence $\epsilon$ has more than two integers.
\end{proof}

By Theorem \ref{firstreduction}, changing every nonzero integer $\epsilon_i$ in $\epsilon$ into $|\epsilon_i|$ copies of consecutive $|\epsilon_i|/\epsilon_i \in \{1,-1\}$ will not change the number of free linear digraphs and circular digraphs. We illustrate this in the next example.

\begin{example} \label{example:motivatesecondred}
Suppose $m=5$. Let $\epsilon = (3,0,-1,-2), \epsilon' = (1,1,1,0,-1,-1,-1)$, and $\epsilon''=(3,-3)$. The digraphs $\Gamma^5_{\epsilon}, \Gamma^5_{\epsilon'}$ and $\Gamma^5_{\epsilon''}$ are as follows:

\begin{figure}[H]
\centering
\begin{tikzpicture}
\tikzset{vertex/.style = {shape=circle,draw,minimum size=1.5em}}
\tikzset{edge/.style = {->,> = latex'}}
\node (aa) at (-1,0) {$$};
\node (bb) at (-1,-0.5) {$$};
\node (cc) at (-1,-1) {$$};
\node (dd) at (-1,-1.5) {$$};
\node (ee) at (-1,-2) {$$};
\node (a) at (0,0) {$x_{0,1}$};
\node (b) at (0,-0.5) {$x_{1,1}$};
\node (c) at (0,-1) {$x_{2,1}$};
\node (d) at (0,-1.5) {$x_{3,1}$};
\node (e) at (0,-2) {$x_{4,1}$};
\node (f) at (2,0) {$0$};
\node (g) at (2,-0.5) {$0$};
\node (h) at (2,-1) {$0$};
\node (i) at (2,-1.5) {$x_{3,2}$};
\node (j) at (2,-2) {$x_{4,2}$};
\node (k) at (4,0) {$0$};
\node (l) at (4,-0.5) {$0$};
\node (m) at (4,-1) {$x_{2,3}$};
\node (n) at (4,-1.5) {$x_{3,3}$};
\node (o) at (4,-2)  {$x_{4,3}$};
\node (p) at (6,0) {$x_{0,4}$};
\node (q) at (6,-0.5) {$x_{1,4}$};
\node (r) at (6,-1) {$x_{2,4}$};
\node (s) at (6,-1.5) {$x_{3,4}$};
\node (t) at (6,-2)  {$x_{4,4}$};
\node (ff) at (7,0) {$$};
\node (gg) at (7,-0.5) {$$};
\node (hh) at (7,-1) {$$};
\node (ii) at (7,-1.5) {$$};
\node (jj) at (7,-2) {$$};
\draw[edge] (a) -- (i) node[sloped,midway,anchor=center, above] {\tiny $3$};
\draw[edge] (b) -- (j) node[sloped,midway,anchor=center, above] {\tiny $3$};
\draw[edge] (g) -- (k) node[sloped,midway,anchor=center, above] {\tiny $-1$};
\draw[edge] (h) -- (l) node[sloped,midway,anchor=center, above] {\tiny $-1$};
\draw[edge] (i) -- (m) node[sloped,midway,anchor=center, above] {\tiny $-1$};
\draw[edge] (j) -- (n) node[sloped,midway,anchor=center, above] {\tiny $-1$};
\draw[edge] (m) -- (p) node[sloped,midway,anchor=center, above] {\tiny $-2$};
\draw[edge] (n) -- (q) node[sloped,midway,anchor=center, above] {\tiny $-2$};
\draw[edge] (o) -- (r) node[sloped,midway,anchor=center, above] {\tiny $-2$};
\draw[edge] (p) -- (ff) node[sloped,midway,anchor=center, above] {\tiny $0$};
\draw[edge] (q) -- (gg) node[sloped,midway,anchor=center, above] {\tiny $0$};
\draw[edge] (r) -- (hh) node[sloped,midway,anchor=center, above] {\tiny $0$};
\draw[edge] (s) -- (ii) node[sloped,midway,anchor=center, above] {\tiny $0$};
\draw[edge] (t) -- (jj) node[sloped,midway,anchor=center, above] {\tiny $0$};
\draw[edge] (aa) -- (a) node[sloped,midway,anchor=center, above] {\tiny $0$};
\draw[edge] (bb) -- (b) node[sloped,midway,anchor=center, above] {\tiny $0$};
\draw[edge] (cc) -- (c) node[sloped,midway,anchor=center, above] {\tiny $0$};
\draw[edge] (dd) -- (d) node[sloped,midway,anchor=center, above] {\tiny $0$};
\draw[edge] (ee) -- (e) node[sloped,midway,anchor=center, above] {\tiny $0$};
\end{tikzpicture}
\caption{Digraph for $\Gamma^5_{\epsilon}$}
\end{figure}

\begin{figure}[H]
\centering
\begin{tikzpicture}
\tikzset{vertex/.style = {shape=circle,draw,minimum size=1.5em}}
\tikzset{edge/.style = {->,> = latex'}}
\node (a1) at (-1,0) {$$};
\node (a2) at (-1,-0.5) {$$};
\node (a3) at (-1,-1) {$$};
\node (a4) at (-1,-1.5) {$$};
\node (a5) at (-1,-2) {$$};
\node (b1) at (0,0) {$x_{0,1}$};
\node (b2) at (0,-0.5) {$x_{1,1}$};
\node (b3) at (0,-1) {$x_{2,1}$};
\node (b4) at (0,-1.5) {$x_{3,1}$};
\node (b5) at (0,-2) {$x_{4,1}$};
\node (c1) at (1.5,0) {$0$};
\node (c2) at (1.5,-0.5) {$x_{1,2}$};
\node (c3) at (1.5,-1) {$x_{2,2}$};
\node (c4) at (1.5,-1.5) {$x_{3,2}$};
\node (c5) at (1.5,-2) {$x_{4,2}$};
\node (d1) at (3,0) {$0$};
\node (d2) at (3,-0.5) {$0$};
\node (d3) at (3,-1) {$x_{2,3}$};
\node (d4) at (3,-1.5) {$x_{3,3}$};
\node (d5) at (3,-2)  {$x_{4,3}$};
\node (e1) at (4.5,0) {$0$};
\node (e2) at (4.5,-0.5) {$0$};
\node (e3) at (4.5,-1) {$0$};
\node (e4) at (4.5,-1.5) {$x_{3,4}$};
\node (e5) at (4.5,-2)  {$x_{4,4}$};
\node (f1) at (6,0) {$0$};
\node (f2) at (6,-0.5) {$0$};
\node (f3) at (6,-1) {$x_{2,5}$};
\node (f4) at (6,-1.5) {$x_{3,5}$};
\node (f5) at (6,-2) {$x_{4,5}$};
\node (g1) at (7.5,0) {$0$};
\node (g2) at (7.5,-0.5) {$x_{1,6}$};
\node (g3) at (7.5,-1) {$x_{2,6}$};
\node (g4) at (7.5,-1.5) {$x_{3,6}$};
\node (g5) at (7.5,-2) {$x_{4,6}$};
\node (h1) at (9,0) {$x_{0,7}$};
\node (h2) at (9,-0.5) {$x_{1,7}$};
\node (h3) at (9,-1) {$x_{2,7}$};
\node (h4) at (9,-1.5) {$x_{3,7}$};
\node (h5) at (9,-2) {$x_{4,7}$};
\node (i1) at (10,0) {$$};
\node (i2) at (10,-0.5) {$$};
\node (i3) at (10,-1) {$$};
\node (i4) at (10,-1.5) {$$};
\node (i5) at (10,-2) {$$};
\draw[edge] (a1) -- (b1) node[sloped,midway,anchor=center, above] {\tiny $0$};
\draw[edge] (a2) -- (b2) node[sloped,midway,anchor=center, above] {\tiny $0$};
\draw[edge] (a3) -- (b3) node[sloped,midway,anchor=center, above] {\tiny $0$};
\draw[edge] (a4) -- (b4) node[sloped,midway,anchor=center, above] {\tiny $0$};
\draw[edge] (a5) -- (b5) node[sloped,midway,anchor=center, above] {\tiny $0$};
\draw[edge] (b1) -- (c2) node[sloped,midway,anchor=center, above] {\tiny $1$};
\draw[edge] (b2) -- (c3) node[sloped,midway,anchor=center, above] {\tiny $1$};
\draw[edge] (b3) -- (c4) node[sloped,midway,anchor=center, above] {\tiny $1$};
\draw[edge] (b4) -- (c5) node[sloped,midway,anchor=center, above] {\tiny $1$};
\draw[edge] (c1) -- (d2) node[sloped,midway,anchor=center, above] {\tiny $1$};
\draw[edge] (c2) -- (d3) node[sloped,midway,anchor=center, above] {\tiny $1$};
\draw[edge] (c3) -- (d4) node[sloped,midway,anchor=center, above] {\tiny $1$};
\draw[edge] (c4) -- (d5) node[sloped,midway,anchor=center, above] {\tiny $1$};
\draw[edge] (d1) -- (e2) node[sloped,midway,anchor=center, above] {\tiny $1$};
\draw[edge] (d2) -- (e3) node[sloped,midway,anchor=center, above] {\tiny $1$};
\draw[edge] (d3) -- (e4) node[sloped,midway,anchor=center, above] {\tiny $1$};
\draw[edge] (d4) -- (e5) node[sloped,midway,anchor=center, above] {\tiny $1$};
\draw[edge] (e2) -- (f1) node[sloped,midway,anchor=center, above] {\tiny $1$};
\draw[edge] (e3) -- (f2) node[sloped,midway,anchor=center, above] {\tiny $1$};
\draw[edge] (e4) -- (f3) node[sloped,midway,anchor=center, above] {\tiny $1$};
\draw[edge] (e5) -- (f4) node[sloped,midway,anchor=center, above] {\tiny $1$};
\draw[edge] (f2) -- (g1) node[sloped,midway,anchor=center, above] {\tiny $1$};
\draw[edge] (f3) -- (g2) node[sloped,midway,anchor=center, above] {\tiny $1$};
\draw[edge] (f4) -- (g3) node[sloped,midway,anchor=center, above] {\tiny $1$};
\draw[edge] (f5) -- (g4) node[sloped,midway,anchor=center, above] {\tiny $1$};
\draw[edge] (g2) -- (h1) node[sloped,midway,anchor=center, above] {\tiny $1$};
\draw[edge] (g3) -- (h2) node[sloped,midway,anchor=center, above] {\tiny $1$};
\draw[edge] (g4) -- (h3) node[sloped,midway,anchor=center, above] {\tiny $1$};
\draw[edge] (g5) -- (h4) node[sloped,midway,anchor=center, above] {\tiny $1$};
\draw[edge] (h1) -- (i1) node[sloped,midway,anchor=center, above] {\tiny $0$};
\draw[edge] (h2) -- (i2) node[sloped,midway,anchor=center, above] {\tiny $0$};
\draw[edge] (h3) -- (i3) node[sloped,midway,anchor=center, above] {\tiny $0$};
\draw[edge] (h4) -- (i4) node[sloped,midway,anchor=center, above] {\tiny $0$};
\draw[edge] (h5) -- (i5) node[sloped,midway,anchor=center, above] {\tiny $0$};
\end{tikzpicture}
\caption{Digraph for $\Gamma^5_{\epsilon'}$}
\end{figure}

\begin{figure}[H]
\centering
\begin{tikzpicture}
\tikzset{vertex/.style = {shape=circle,draw,minimum size=1.5em}}
\tikzset{edge/.style = {->,> = latex'}}
\node (aa) at (-1,0) {$$};
\node (bb) at (-1,-0.5) {$$};
\node (cc) at (-1,-1) {$$};
\node (dd) at (-1,-1.5) {$$};
\node (ee) at (-1,-2) {$$};
\node (a) at (0,0) {$x_{0,1}$};
\node (b) at (0,-0.5) {$x_{1,1}$};
\node (c) at (0,-1) {$x_{2,1}$};
\node (d) at (0,-1.5) {$x_{3,1}$};
\node (e) at (0,-2) {$x_{4,1}$};
\node (f) at (2,0) {$x_{0,2}$};
\node (g) at (2,-0.5) {$x_{1,2}$};
\node (h) at (2,-1) {$x_{2,2}$};
\node (i) at (2,-1.5) {$x_{3,2}$};
\node (j) at (2,-2) {$x_{4,2}$};
\node (k) at (3,0) {$$};
\node (l) at (3,-0.5) {$$};
\node (m) at (3,-1) {$$};
\node (n) at (3,-1.5) {$$};
\node (o) at (3,-2)  {$$};
\draw[edge] (a) -- (f) node[sloped,midway,anchor=center, above] {\tiny $0$};
\draw[edge] (b) -- (g) node[sloped,midway,anchor=center, above] {\tiny $0$};
\draw[edge] (f) -- (k) node[sloped,midway,anchor=center, above] {\tiny $0$};
\draw[edge] (g) -- (l) node[sloped,midway,anchor=center, above] {\tiny $0$};
\draw[edge] (h) -- (m) node[sloped,midway,anchor=center, above] {\tiny $0$};
\draw[edge] (i) -- (n) node[sloped,midway,anchor=center, above] {\tiny $0$};
\draw[edge] (j) -- (o) node[sloped,midway,anchor=center, above] {\tiny $0$};
\draw[edge] (aa) -- (a) node[sloped,midway,anchor=center, above] {\tiny $0$};
\draw[edge] (bb) -- (b) node[sloped,midway,anchor=center, above] {\tiny $0$};
\draw[edge] (cc) -- (c) node[sloped,midway,anchor=center, above] {\tiny $0$};
\draw[edge] (dd) -- (d) node[sloped,midway,anchor=center, above] {\tiny $0$};
\draw[edge] (ee) -- (e) node[sloped,midway,anchor=center, above] {\tiny $0$};
\end{tikzpicture}
\caption{Digraph for $\Gamma^5_{\epsilon''}$}
\end{figure}
Clearly $c(\Gamma^5_{\epsilon}) = c(\Gamma^5_{\epsilon'}) = 2$ and $\ell(\Gamma^5_{\epsilon}) = \ell(\Gamma^5_{\epsilon'})=3$. This is guaranteed by Theorem \ref{firstreduction}. We also observe that $c(\Gamma^5_{\epsilon''})=2$ and $\ell(\Gamma^5_{\epsilon''}) = 3$. This suggests that we can also remove all the zeroes in $\epsilon$ and it motivates the second reduction.
\end{example}

\subsection{Second Reduction} \label{subs:secondreduction}
\begin{theorem}[Second Reduction] \label{secondreduction}
Let $\epsilon = (\epsilon_1, \epsilon_2, \dots, \epsilon_s)$ be a circular sequence such that not all $\epsilon_i$'s are zeroes. If $\epsilon'$ is the circular sequence constructed from $\epsilon$ by removing all zeroes, then we have $c(\Gamma^m_{\epsilon}) = c(\Gamma^m_{\epsilon'})$ and $\ell(\Gamma^m_{\epsilon}) = \ell(\Gamma^m_{\epsilon'})$.
\end{theorem}
\begin{proof}
To prove this theorem, we can assume that $\epsilon_i \in \{\pm 1, 0\}$ by Theorem \ref{firstreduction} for all $i \in I_s$. If $\epsilon$ does not contain any zeroes, there is nothing to prove. Without loss of generality, we can assume that $\epsilon_2 = 0$. Recall that $\Gamma^m$ is the union of $\Gamma^m_{\epsilon_1, \epsilon_2}$ and $\Gamma^m_{\epsilon_2, \epsilon_3}$, and ${\Gamma'}^m = \Gamma^m_{\epsilon_1, \epsilon_3}$. Let $\epsilon'' = (\epsilon_1, \epsilon_3, \dots, \epsilon_s)$. If we replace $\Gamma^m$ by ${\Gamma'}^m$ in $\Gamma^m_{\epsilon}$, then we get $\Gamma^m_{\epsilon''}$. We compare $\Gamma^m$ and ${\Gamma'}^m$ in four cases as follows.
\medskip
\begin{case}
\item $(\epsilon_1, \epsilon_2, \epsilon_3) \in \{(-1,0,1), (0,0,0), (-1,0,0), (0,0,1)\}.$
\begin{figure}[H]
\centering
\begin{tikzpicture}
\tikzset{vertex/.style = {shape=circle,draw,minimum size=1.5em}}
\tikzset{edge/.style = {->,> = latex'}}
\node (a) at (0,0) {$x_{0,1}$};
\node (b) at (0,-0.5) {$\vdots$};
\node (c) at (0,-1) {$x_{m-1,1}$};
\node (d) at (2,0) {$x_{0,2}$};
\node (e) at (2,-0.5) {$\vdots$};
\node (f) at (2,-1) {$x_{m-1,2}$};
\node (g) at (4,0) {$x_{0,3}$};
\node (h) at (4,-0.5) {$\vdots$};
\node (i) at (4,-1) {$x_{m-1,3}$};
\draw[edge] (a) -- (d) node[sloped,midway,anchor=center, above] {\tiny $0$};
\draw[edge] (d) -- (g) node[sloped,midway,anchor=center, above] {\tiny $0$};
\draw[edge] (c) -- (f) node[sloped,midway,anchor=center, above] {\tiny $0$};
\draw[edge] (f) -- (i) node[sloped,midway,anchor=center, above] {\tiny $0$};
\node (aa) at (7,0) {$x_{0,1}$};
\node (bb) at (7,-0.5) {$\vdots$};
\node (cc) at (7,-1) {$x_{m-1,1}$};
\node (dd) at (10,0) {$x_{0,3}$};
\node (ee) at (10,-0.5) {$\vdots$};
\node (ff) at (10,-1) {$x_{m-1,3}$};
\draw[edge] (aa) -- (dd) node[sloped,midway,anchor=center, above] {\tiny $0$};
\draw[edge] (cc) -- (ff) node[sloped,midway,anchor=center, above] {\tiny $0$};
\end{tikzpicture}
\caption{Digraphs for $\Gamma^m$ and ${\Gamma'}^m$ in Case (1)} 
\end{figure}

\item $(\epsilon_1, \epsilon_2, \epsilon_3) \in \{(-1,0,-1), (0,0,-1)\}$.
\begin{figure}[H]
\centering
\begin{tikzpicture}
\tikzset{vertex/.style = {shape=circle,draw,minimum size=1.5em}}
\tikzset{edge/.style = {->,> = latex'}}
\node (a) at (0,0) {$0$};
\node (b) at (0,-0.5) {$x_{1,1}$};
\node (c) at (0,-1) {$\vdots$};
\node (d) at (0,-1.5) {$x_{m-1,1}$};
\node (e) at (2,0) {$0$};
\node (f) at (2,-0.5) {$x_{1,2}$};
\node (g) at (2,-1) {$\vdots$};
\node (h) at (2,-1.5) {$x_{m-1,2}$};
\node (i) at (4,0) {$x_{0,3}$};
\node (j) at (4,-0.5) {$\vdots$};
\node (k) at (4,-1) {$x_{m-2,3}$};
\node (l) at (4,-1.5) {$x_{m-1,3}$};
\draw[edge] (a) -- (e) node[sloped,midway,anchor=center, above] {\tiny $0$};
\draw[edge] (b) -- (f) node[sloped,midway,anchor=center, above] {\tiny $0$};
\draw[edge] (d) -- (h) node[sloped,midway,anchor=center, above] {\tiny $0$};
\draw[edge] (f) -- (i) node[sloped,midway,anchor=center, above] {\tiny $-1$};
\draw[edge] (h) -- (k) node[sloped,midway,anchor=center, above] {\tiny $-1$};
\node (aa) at (7,0) {$0$};
\node (bb) at (7,-0.5) {$x_{1,1}$};
\node (cc) at (7,-1) {$\vdots$};
\node (dd) at (7,-1.5) {$x_{m-1,1}$};
\node (ee) at (10,0) {$x_{0,3}$};
\node (ff) at (10,-0.5) {$\vdots$};
\node (gg) at (10,-1) {$x_{m-2,3}$};
\node (hh) at (10,-1.5) {$x_{m-1,3}$};
\draw[edge] (bb) -- (ee) node[sloped,midway,anchor=center, above] {\tiny $-1$};
\draw[edge] (dd) -- (gg) node[sloped,midway,anchor=center, above] {\tiny $-1$};
\end{tikzpicture}
\caption{Digraphs for $\Gamma^m$ and ${\Gamma'}^m$ in Case (2)} 
\end{figure}

\item $(\epsilon_1, \epsilon_2, \epsilon_3) \in \{(1,0,0), (1,0,1)\}$.
\begin{figure}[H]
\centering
\begin{tikzpicture}
\tikzset{vertex/.style = {shape=circle,draw,minimum size=1.5em}}
\tikzset{edge/.style = {->,> = latex'}}
\node (a) at (0,0) {$x_{0,1}$};
\node (b) at (0,-0.5) {$\vdots$};
\node (c) at (0,-1) {$x_{m-2,1}$};
\node (d) at (0,-1.5) {$x_{m-1,1}$};
\node (e) at (2,0) {$0$};
\node (f) at (2,-0.5) {$x_{1,2}$};
\node (g) at (2,-1) {$\vdots$};
\node (h) at (2,-1.5) {$x_{m-1,2}$};
\node (i) at (4,0) {$0$};
\node (j) at (4,-0.5) {$x_{1,3}$};
\node (k) at (4,-1) {$\vdots$};
\node (l) at (4,-1.5) {$x_{m-1,3}$};
\draw[edge] (a) -- (f) node[sloped,midway,anchor=center, above] {\tiny $1$};
\draw[edge] (c) -- (h) node[sloped,midway,anchor=center, above] {\tiny $1$};
\draw[edge] (e) -- (i) node[sloped,midway,anchor=center, above] {\tiny $0$};
\draw[edge] (f) -- (j) node[sloped,midway,anchor=center, above] {\tiny $0$};
\draw[edge] (h) -- (l) node[sloped,midway,anchor=center, above] {\tiny $0$};
\node (aa) at (7,0) {$x_{0,1}$};
\node (bb) at (7,-0.5) {$\vdots$};
\node (cc) at (7,-1) {$x_{m-2,1}$};
\node (dd) at (7,-1.5) {$x_{m-1,1}$};
\node (ee) at (10,0) {$0$};
\node (ff) at (10,-0.5) {$x_{1,3}$};
\node (gg) at (10,-1) {$\vdots$};
\node (hh) at (10,-1.5) {$x_{m-1,3}$};
\draw[edge] (aa) -- (ff) node[sloped,midway,anchor=center, above] {\tiny $1$};
\draw[edge] (cc) -- (hh) node[sloped,midway,anchor=center, above] {\tiny $1$};
\end{tikzpicture}
\caption{Digraphs for $\Gamma^m$ and ${\Gamma'}^m$ in Case (3)} 
\end{figure}

\item $(\epsilon_1, \epsilon_2, \epsilon_3) = (1,0,-1)$.
\begin{figure}[H]
\centering
\begin{tikzpicture}
\tikzset{vertex/.style = {shape=circle,draw,minimum size=1.5em}}
\tikzset{edge/.style = {->,> = latex'}}
\node (a) at (0,0) {$x_{0,1}$};
\node (b) at (0,-0.5) {$\vdots$};
\node (c) at (0,-1) {$x_{m-2,1}$};
\node (d) at (0,-1.5) {$x_{m-1,1}$};
\node (e) at (2,0) {$0$};
\node (f) at (2,-0.5) {$x_{1,2}$};
\node (g) at (2,-1) {$\vdots$};
\node (h) at (2,-1.5) {$x_{m-1,2}$};
\node (i) at (4,0) {$x_{0,3}$};
\node (j) at (4,-0.5) {$\vdots$};
\node (k) at (4,-1) {$x_{m-2,3}$};
\node (l) at (4,-1.5) {$x_{m-1,3}$};
\draw[edge] (a) -- (f) node[sloped,midway,anchor=center, above] {\tiny $1$};
\draw[edge] (c) -- (h) node[sloped,midway,anchor=center, above] {\tiny $1$};
\draw[edge] (f) -- (i) node[sloped,midway,anchor=center, above] {\tiny $-1$};
\draw[edge] (h) -- (k) node[sloped,midway,anchor=center, above] {\tiny $-1$};
\node (aa) at (7,0) {$x_{0,1}$};
\node (bb) at (7,-0.5) {$\vdots$};
\node (cc) at (7,-1) {$x_{m-2,1}$};
\node (dd) at (7,-1.5) {$x_{m-1,1}$};
\node (ee) at (10,0) {$x_{0,3}$};
\node (ff) at (10,-0.5) {$\vdots$};
\node (gg) at (10,-1) {$x_{m-2,3}$};
\node (hh) at (10,-1.5) {$x_{m-1,3}$};
\draw[edge] (aa) -- (ee) node[sloped,midway,anchor=center, above] {\tiny $0$};
\draw[edge] (cc) -- (gg) node[sloped,midway,anchor=center, above] {\tiny $0$};
\end{tikzpicture}
\caption{Digraphs for $\Gamma^m$ and ${\Gamma'}^m$ in Case (4)} 
\end{figure}
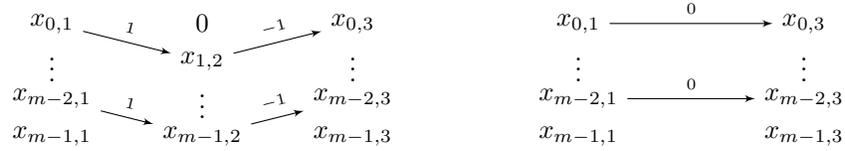
\end{case}
In each of the four cases above, it is easy to see that if we replace $\Gamma^m$ with ${\Gamma'}^m$ in $\Gamma^m_{\epsilon}$, the number of free linear digraphs and the number of circular digraphs do not change. By applying this process repeatedly, we can remove all the zeroes in the $\epsilon$ sequence to get $\epsilon'$ such that $c(\Gamma^m_{\epsilon}) = c(\Gamma^m_{\epsilon'})$ and $\ell(\Gamma^m_{\epsilon}) = \ell(\Gamma^m_{\epsilon'})$.
\end{proof}

\section{Dieudonn\'e Modules and the Main Result} \label{section:dm}

Let $\epsilon = (\epsilon_1, \dots, \epsilon_s)$ be a circular sequence of integers of length $s$. For every positive integer $t > s$, let $t' \in I_s $ be such that $t \equiv t'$ modulo $s$; set $\epsilon_{t} = \epsilon_{t'}$. In this section, we develop a more direct algorithm to compute $\ell(\Gamma^m_{\epsilon})$ and $c(\Gamma^m_{\epsilon})$ from $\epsilon$ without having to consider the digraph $\Gamma^m_{\epsilon}$. If $\epsilon_1 = \cdots = \epsilon_s = 0$, then $\ell(\Gamma_{\epsilon}^m) = 0$ and $c(\Gamma_{\epsilon}^m)=m$. Suppose now that not all $e_i$'s are zeroes. 

\subsection{Calculation of $\ell(\Gamma^m_{\epsilon})$}
By Theorems \ref{firstreduction} and \ref{secondreduction}, we assume in this subsection that in fact we have $\epsilon_i \in \{1,-1\}$ for all $i \in I_s$. We denote by 
\[x_{l_1,t_1}, x_{l_2,t_2}, \cdots, x_{l_n,t_n}\] 
a free linear digraph of $\Gamma^m_{\epsilon}$, where $x_{l_1,t_1}$ is the origin and $x_{l_n,t_n}$ is the terminal. 

\begin{lemma} \label{lemma:pdivbasic1}
We have $\epsilon_{t_1} = -1$, $\epsilon_{t_n}=1$, and $l_1=l_n = m-1$.
\end{lemma}
\begin{proof}
The only possible way for $x_{l_1,t_1}, x_{l_2,t_2}, \cdots, x_{l_n,t_n}$ to have a nonzero origin in $\Gamma^m_{\epsilon}$ is that $l_1 = m-1$, and $(\epsilon_{t_1-1}, \epsilon_{t_1}) \in \{(-1,-1), (1,-1)\}$, where $\epsilon_{t_1-1}$ is the integer before $\epsilon_{t_1}$ in $\epsilon$ (see Cases (1) and (8) of Subsection \ref{subsection:weighteddigraphs}). For the same reason, the only possible way for $x_{l_1,t_1}, x_{l_2,t_2}, \cdots, x_{l_n,t_n}$  to have a nonzero terminal in $\Gamma^m_{\epsilon}$ is that $l_n = m-1$, and $(\epsilon_{t_{n}}, \epsilon_{t_n+1}) \in \{(1,-1), (1,1)\}$, where $\epsilon_{t_n+1}$ is the integer after $\epsilon_{t_n}$ in $\epsilon$ (see Cases (1) and (10) of Subsection \ref{subsection:weighteddigraphs}). The lemma follows from the last two sentences.
\end{proof}

\begin{corollary} \label{corollary:freelinearsinglevertex}
Let $\epsilon=(\epsilon_1, \dots, \epsilon_s)$ be a circular sequence such that $\epsilon_i \in \{1, -1\}$ for all $i \in I_s$. No free linear digraph of $\Gamma^m_{\epsilon}$ can have just one single vertex. 
\end{corollary}
\begin{proof}
This is an easy consequence of Lemma \ref{lemma:pdivbasic1}.
\end{proof}

\begin{lemma} \label{lemma:pdivbasic2}
For each $i \in I_n$, we have the following formula for the first subscript of $x_{l_i,t_i}$:
\begin{equation*}
l_i = 
\begin{cases}
l_1+\sum_{j=1}^i \epsilon_{t_j}+1 = m-1+\sum_{j=1}^i \epsilon_{t_j}+1 & \text{if } \epsilon_{t_i} = -1,\\
l_1+\sum_{j=1}^i \epsilon_{t_j}+1 = m-1+\sum_{j=1}^i \epsilon_{t_j} & \text{if } \epsilon_{t_i}=1.
\end{cases}
\end{equation*}
Specifically, we have $\sum_{j=1}^n \epsilon_{t_j} = 0$. For every $i \in I_{n-1}$, we have $-m \leq \sum_{j=1}^i \epsilon_{t_j} < 0$.
\end{lemma}
\begin{proof}
For each $i \in I_n$,
\begin{enumerate}
\item if $(\epsilon_{t_i}, \epsilon_{t_{i+1}}) \in \{(-1,-1), (1,1)\}$, then $l_{i+1} = l_{i}+\epsilon_{t_{i+1}}$;
\item if $(\epsilon_{t_i}, \epsilon_{t_{i+1}}) \in \{(-1,1), (1,-1)\}$, then $l_{i+1} = l_{i}$.
\end{enumerate}
In other words, if the sign remains the same from $\epsilon_{t_i}$ to $\epsilon_{t_{i+1}}$, then $l_{i+1}-l_i = \epsilon_{t_{i+1}}$; if the sign changes from $\epsilon_{t_i}$ to $\epsilon_{t_{i+1}}$, then $l_{i+1} - l_{i} = 0$. 

If $\epsilon_{t_i}=-1$, then the sign changes an even number of times from $\epsilon_{t_1}$ to $\epsilon_{t_i}$, where half of the changes are from $-1$ and $1$ and the other half of the changes are from $1$ to $-1$. As a result, 
\[l_i - l_1 = \sum_{j=2}^i \epsilon_{t_j} = \sum_{j=1}^i \epsilon_{t_j}+1.\]

If $\epsilon_{t_i}=1$, then the sign changes an odd number of times from $\epsilon_{t_1}$ to $\epsilon_{t_i}$, where the number of changes from $-1$ to $1$ is one more than the number of changes from $1$ to $-1$. As a result, 
\[l_i - l_1 = \sum_{j=2}^i \epsilon_{t_j}-1 = \sum_{j=1}^i \epsilon_j.\]
As $l_1 = m-1$ by Lemma \ref{lemma:pdivbasic1}, this proves the formula for $l_i$ for all $1 \leq i \leq n$. As $\epsilon_{t_n}=1$ and $l_n=m-1$, we get that $\sum_{j=1}^n \epsilon_{t_j} = 0$.

For every $i \in I_n$, we have $l_i \in \{0, \dots, m-1\}$. Hence $-m \leq \sum_{j=1}^i \epsilon_{t_j} \leq 0$ and moreover, if $\sum_{j=1}^i \epsilon_{t_j} = 0$, then $\epsilon_{t_i}=1$. If there exists $1 < i < n$ such that $\epsilon_{t_i} = 1$ and $\sum_{j=1}^i \epsilon_{t_j} = 0$, then $l_i = m-1$. As $\epsilon_{t_i} = 1$ and $l_i = m-1$, we get that $x_{l_i,t_i} = x_{m-1,t_i}$ is the terminal $x_{l_n,t_n} = x_{m-1,t_n}$ of the free linear digraph and thus $i=n$, a contradiction. Therefore $\sum_{j=1}^i \epsilon_{t_j} < 0$ for all $i \in I_{n-1}$.
\end{proof}

\begin{corollary}
The length $n$ of every free linear digraph of $\Gamma_{\epsilon}^m$ is no longer than the length of $\epsilon$, namely, $n \leq s$.
\end{corollary}
\begin{proof}
We show that the assumption $n>s$ leads to a contradiction. Without loss of generality, we can assume that $t_1=1$ due to the circular nature of $\epsilon$. Let $n = qs+r$, where $q, r$ are integers such that $q \geq 1$ and  $0 \leq r < s$. By Lemma \ref{lemma:pdivbasic2}, we know that $\sum_{j=1}^s \epsilon_j =: \epsilon_0 < 0$. As 
\[0 = \sum_{j=1}^n \epsilon_{t_j} = q\epsilon_0+\sum_{j=n-r+1}^n\epsilon_{t_j},\]
we get that $r > 0$ and \[\sum_{j=n-r+1}^n \epsilon_j = -q\epsilon_0>0.\] On the other hand, we have 
\[\sum_{j=n-r+1}^n \epsilon_{t_j} = \sum_{j=qs+1}^{qs+r} \epsilon_{t_j} = \sum_{j=1}^r \epsilon_{t_j} < 0,\] a contradiction. Hence $n \leq s$.
\end{proof}

\begin{definition} \label{definition:freelinearsegment}
Let $\epsilon = (\epsilon_1, \epsilon_2, \dots, \epsilon_s)$ be a circular sequence of integers such that $\epsilon_i \in \{1, -1\}$ for all $i \in I_s$. A segment $\epsilon_a, \epsilon_{a+1}, \dots, \epsilon_b$ of $\epsilon$ is called a free linear segment of level $\lambda \geq 1$ if it satisfies the following four conditions:
\begin{enumerate}
\item we have $\epsilon_a=-1$ and $\epsilon_b=1$;
\item we have $\sum_{j=a}^b \epsilon_j = 0$;
\item for all $i \in \{a, a+1, \dots, b-1\}$, we have $-\lambda \leq \sum_{j=a}^i \epsilon_j < 0$;
\item there exists $i_1 \in \{a, a+1, \dots, b-1\}$ such that we have $\sum_{j=a}^{i_1} \epsilon_j = -\lambda$.
\end{enumerate}
We denote by $a_{\lambda}(\epsilon)$ the number of free linear segments of level $\lambda$ in $\epsilon$.
\end{definition}

\begin{remark} \label{remark:0dontmatterina}
Definition \ref{definition:freelinearsegment} still makes sense even if we allow $\epsilon$ to contain $0$'s as $0$'s do not change sums of the form $\sum_{j=a}^i \epsilon_j$ (with $a \leq i \leq b$) and thus they do not change the level of a free linear segment. Also adding $0$'s does not change the below results in this subsection due to Theorem \ref{secondreduction}.
\end{remark}

\begin{proposition} \label{proposition:freelinearsegmentcontains}
Each free linear segment of level $\lambda \geq 2$ contains at least one free linear segment of level $\lambda-1$.
\end{proposition}
\begin{proof}
Let $\epsilon_a, \epsilon_{a+1}, \dots, \epsilon_b$ be a free linear segment of level $\lambda \geq 2$ with $1 \leq a < b \leq s$. Let $a < i_1 < b$ be the smallest integer such that $\sum_{j=a}^{i_1} \epsilon_j = -\lambda$ and $\epsilon_{i_1} = -1$. As $\lambda \geq 2$, we know that $\epsilon_{a+1}=-1$ and $\epsilon_{b-1} = 1$. Therefore $\sum_{j=a+1}^{i_1} \epsilon_j= -(\lambda-1)$. We can then find $a+1 \leq i_0 \leq i_1$ such that $\epsilon_{i_0} = -1$, $\sum_{j=i_0}^{i_1} \epsilon_j = -(\lambda-1)$ and $\sum_{j=i_0}^i \epsilon_j < 0$ for all $i_0 \leq i < i_1$. As 
\[\sum_{j=i_0}^{b-1} \epsilon_j = \sum_{j=i_0}^{i_1} \epsilon_j + (\sum_{j=a+1}^{b-1} \epsilon_j - \sum_{j=a+1}^{i_1}\epsilon_j)= -(\lambda-1)+(0-(-(\lambda-1)))=0,\]
 we can find $i_1 < i_2 \leq b-1$ such that $\epsilon_{i_2} = 1$, $\sum_{j=i_0}^{i_2} \epsilon_j = 0$, and $\sum_{j=i_0}^i \epsilon_j < 0$ for all $i_0 \leq i < i_2$. Therefore $\epsilon_{i_0}, \dots, \epsilon_{i_2}$ is a free linear segment of level $\lambda-1$.
\end{proof}

\begin{corollary} \label{corollary:freelinearsegmentcontains}
For $\lambda \geq 2$, $a_{\lambda}(\epsilon) \leq a_{\lambda-1}(\epsilon)$.
\end{corollary}
\begin{proof}
This is an easy consequence of Proposition \ref{proposition:freelinearsegmentcontains}.
\end{proof}

\begin{lemma} \label{lemma:pdivbasic3}
Each free linear segment $\epsilon_{t_1}, \dots, \epsilon_{t_n}$ of level $\lambda \in I_m$ in $\epsilon$ corresponds to a unique free linear digraph of the form $x_{l_1,t_1}, x_{l_2,t_2}, \cdots, x_{l_n,t_n}$ in $\Gamma^m_{\epsilon}$, where $l_i$ is given by the formula:
\begin{equation*}
l_i = 
\begin{cases}
l_1+\sum_{j=1}^i \epsilon_{t_j}+1 = m-1+\sum_{j=1}^i \epsilon_{t_j}+1 & \text{if } \epsilon_{t_i} = -1,\\
l_1+\sum_{j=1}^i \epsilon_{t_j}+1 = m-1+\sum_{j=1}^i \epsilon_{t_j} & \text{if } \epsilon_{t_i}=1.
\end{cases}
\end{equation*}
Furthermore, we have $\mathrm{max}\{l_i \mid i \in I_n \} = l_1=l_n=m-1$ and $\mathrm{min}\{l_i \mid i \in I_n\} = m-\lambda \geq 0$.
\end{lemma}
\begin{proof}
As $\epsilon_{t_1} = -1$, we know that the only vertex in $\Gamma^m_{\epsilon}$ of the form $x_{l_1,t_1}$ that is not the target of any edge is $x_{m-1,t_1}$. Similarly, as $\epsilon_{t_n} = 1$, we know that the only vertex in $\Gamma^m_{\epsilon}$ of the form $x_{l_n,t_n}$ that is not the source of any edge is $x_{m-1,t_n}$. For every $i \in \{2, \dots, n-1\}$, if $\epsilon_{t_i} = -1$, then $-\lambda \leq \sum_{j=1}^i\epsilon_{t_j} \leq -2$ and thus $m-\lambda \leq l_i \leq m-2$; if $\epsilon_{t_i} = 1$, then $-\lambda+1 \leq \sum_{j=1}^i \epsilon_{t_j} \leq -1$ and thus $m-\lambda \leq l_i \leq m-2$. Hence for all $i \in \{2, \dots, n-1\}$, we have $0 \leq m-\lambda \leq l_i < m-1$. Let $i_1 \in \{2, \dots, n-1\}$ be such that $\sum_{j=1}^{i_1} \epsilon_{t_j} = -\lambda$. As $\epsilon_{t_{i_1}} = -1$, we know that $l_{t_{i_1}} = m-\lambda$. This means that $\mathrm{min}\{l_i \mid i \in I_n\} = m-\lambda$. On the other hand, we know that $l_{t_1} = l_{t_n}=m-1$. This means that $\mathrm{max}\{l_i \mid i \in I_n \} = m-1$.

To show that $x_{l_1,t_1}, x_{l_2,t_2}, \cdots, x_{l_n,t_n}$ is the unique free linear digraph of $\Gamma^m_{\epsilon}$ corresponding to the free linear segment $\epsilon_{t_1}, \dots, \epsilon_{t_n}$ of level $\lambda$ in $\epsilon$, it remains to show that for every $i \in I_n$, there is an edge whose source is $x_{l_i,t_i}$ and whose target is $x_{l_{i+1},t_{i+1}}$. We discuss this in four cases:
\begin{enumerate}
\item If $(\epsilon_{t_i}, \epsilon_{t_{i+1}}) = (-1,-1)$, then $l_{i+1} - l_i = -1$ and $-\lambda+1 \leq \sum_{j=1}^i \epsilon_{t_j} \leq -1$. Hence, 
\[l_i  = m+\sum_{j=1}^i \epsilon_{t_j} \in \{m-\lambda+1, m-\lambda+2, \dots, m-1\}.\]
According to the construction of $\Gamma^m_{\epsilon}$ (see Case (8) of Subsection \ref{subsection:weighteddigraphs}), there is an edge whose source is $x_{l_i,t_i}$ and whose target is $x_{l_{i+1},t_{i+1}}$.
\item If $(\epsilon_{t_i}, \epsilon_{t_{i+1}}) = (1,1)$, then $l_{i+1} - l_i = 1$ and $-\lambda+2 \leq  \sum_{j=1}^{i+1} \epsilon_{t_j} \leq 0$. Hence,
\[l_{i+1} = m-1+\sum_{j=1}^{i+1} \epsilon_{t_j} \in \{m-\lambda+1, m-\lambda-2, \dots, m-1\}.\]
According to the construction of $\Gamma^m_{\epsilon}$ (see Case (10) of Subsection \ref{subsection:weighteddigraphs}), there is an edge whose source is $x_{l_i,t_i}$ and whose target is $x_{l_{i+1},t_{i+1}}$.
\item If $(\epsilon_{t_i}, \epsilon_{t_{i+1}}) = (1,-1)$, then $l_{i+1} - l_i = 0$ and $-\lambda +1 \leq \sum_{j=1}^{i} \epsilon_{t_j} \leq -1$ (so $\lambda>1$). Hence 
\[l_i = l_{i+1} = m-1+\sum_{j=1}^i \epsilon_{t_j} \in \{m-\lambda, m-\lambda+1, \dots, m-2 \}.\]
According to the construction of $\Gamma^m_{\epsilon}$ (see Case (1) of Subsection \ref{subsection:weighteddigraphs}), there is an edge whose source is $x_{l_i,t_i}$ and whose target is $x_{l_{i+1},t_{i+1}}$ (even if $\lambda=m$).
\item If $(\epsilon_{t_i}, \epsilon_{t_{i+1}}) = (-1,1)$, then $l_{i+1} - l_i=0$ and $-\lambda+1 \leq \sum_{j=1}^{i+1} \epsilon_{t_j} \leq 0$. Hence 
\[l_i = l_{i+1} = m-1+\sum_{j=1}^{i+1} \epsilon_{t_j} \in \{m-\lambda, m-\lambda+1, \dots, m-1\}.\]
According to the construction of $\Gamma^m_{\epsilon}$ (see Case (2) of Subsection \ref{subsection:weighteddigraphs}), there is an edge whose source is $x_{l_i,t_i}$ and whose target is $x_{l_{i+1},t_{i+1}}$ (even if $\lambda =m$). \qedhere
\end{enumerate}
\end{proof}

\begin{lemma} \label{lemma:freelineartoolow}
Let $\epsilon_{t_1}, \epsilon_{t_2}, \dots, \epsilon_{t_n}$ be a free linear segment of level $\lambda > m$ in $\epsilon$. Let $a \in \{2, \dots, n-1\}$ be the smallest integer such that \[\sum_{i=1}^{a} \epsilon_{t_i}=-m-1.\] If $x_{l_1,t_1}, x_{l_2,t_2}, \dots, x_{l_b,t_b}$ is a linear digraph of $\Gamma_{\epsilon}^m$ with $b \in \{2, \dots, n \}$, then $b < a < n$. Thus, $\Gamma_{\epsilon}^m$ does not contain a free linear digraph of the form $x_{l_1,t_1}, x_{l_2,t_2}, \cdots, x_{l_n,t_n}$.
\end{lemma}
\begin{proof}
For $i \in \{1, \dots, b-1\}$, the difference $l_{i+1}-l_i$ is as in the four cases of the proof of Lemma \ref{lemma:pdivbasic3}. Thus as in the mentioned proof we get that for $i\in\{1,\ldots,b\}$ we have
\begin{equation*}
l_i =
\begin{cases}
l_1+\sum_{j=1}^i \epsilon_{t_j}+1 & \text{if } \epsilon_{t_i} = -1,\\
l_1+\sum_{j=1}^i \epsilon_{t_j} & \text{if } \epsilon_{t_i}=1.
\end{cases}
\end{equation*}
In particular, the smallest property of $a$ implies that $\epsilon_{t_a}=-1$, and therefore if $b \geq a$ we get that 
\[l_a = l_1+ \sum_{i=1}^a \epsilon_{t_i} + 1 = m-1 +(-m-1)+1 = -1 <0,\]
a contradiction to $l_a \in \{0, 1, \dots, m-1\}$. Thus the lemma holds.
\end{proof}

\begin{theorem} \label{theorem:freelinearcal}
Let $\epsilon = (\epsilon_1, \epsilon_2, \dots, \epsilon_s)$ be a circular sequence of integers such that $\epsilon_i \in \{1, -1\}$ for all $i \in I_s$. We have $\ell(\Gamma^m_{\epsilon}) = \sum_{\lambda=1}^m a_{\lambda}(\epsilon)$.
\end{theorem}
\begin{proof}
By Lemma \ref{lemma:pdivbasic3}, we know that $\ell(\Gamma^m_{\epsilon}) \geq \sum_{\lambda=1}^m a_{\lambda}(\epsilon)$.

Let $x_{l_1,t_1}, x_{l_2,t_2}, \cdots, x_{l_n,t_n}$ be a free linear digraph of $\Gamma^m_{\epsilon}$. By Lemma \ref{lemma:pdivbasic1}, we know that $\epsilon_{t_1} = -1$ and $\epsilon_{t_n}=1$. By Lemma \ref{lemma:pdivbasic2}, we know that $\sum_{j=1}^n \epsilon_{t_j} = 0$. Let
\[S := \left\{-\sum_{j=1}^i \epsilon_{t_j} \; | \; 1 \leq i <n \right\} \quad \textrm{ and } \quad \lambda:= \textrm{max}(S).\]
By Lemma \ref{lemma:pdivbasic2}, we know that every element in $S$ must be positive and $1 \leq \lambda \leq m$. Thus $\epsilon_{t_1}, \dots, \epsilon_{t_n}$ is a free linear segment of level $\lambda$. From this and the uniqueness part of Lemma \ref{lemma:pdivbasic3}, we get that $\ell(\Gamma^m_{\epsilon}) \leq \sum_{\lambda=1}^m a_{\lambda}(\epsilon)$.
\end{proof}

\subsection{Calculation of $c(\Gamma^m_{\epsilon})$}
If $c(\Gamma^m_{\epsilon}) >0$, then $\sum_{i=1}^s \epsilon_i = 0$ by Proposition \ref{proposition:circularzero}. In this subsection, we assume that $\sum_{i=1}^s \epsilon_i = 0$ unless mentioned otherwise. Let 
\[\lambda_{\epsilon} = \textrm{max}\left\{|\sum_{i=a}^b \epsilon_i | \, \biggr\rvert \, 1 \leq a \leq s, a \leq b \leq a+s-1\right\} .\]
Note that $\lambda_{\epsilon} = 0$ if and only if $\epsilon_i = 0$ for every $i \in I_s$. If $\lambda_{\epsilon} \geq 1$, by a circular rearrangement of $\epsilon$, there exists $1 \leq s' < s$ such that $\epsilon_{s'}, \dots, \epsilon_{s}$ is a free linear segment of level $\lambda_{\epsilon}$.

\begin{lemma} \label{lemma:breakupfreelinear}
Let $\epsilon = (\epsilon_1, \epsilon_2, \dots, \epsilon_s)$ be a circular sequence such that $\epsilon_j \in \{1,-1\}$ for every $j \in I_s$ and $\sum_{j=1}^s \epsilon_j = 0$. Then there exists a decomposition 
\[\epsilon = (\widehat{\epsilon}_t, \widehat{\epsilon}_{t-1}, \cdots, \widehat{\epsilon}_1)\]
such that for each $i \in I_t$, $\widehat{\epsilon}_i$ is a free linear segment of $\epsilon$ of level $\lambda_i$ with $1 \leq \lambda_i \leq \lambda_{\epsilon}$.
\end{lemma}
\begin{proof}
To prove the lemma it suffices to find \[1 =: s_t < s_{t-1} < s_{t-2} < \cdots < s_1 < s_0:= s\] such that for all $2 \leq i \leq t$, 
\[\widehat{\epsilon}_i = (\epsilon_{s_i}, \epsilon_{s_i+1}, \dots, \epsilon_{s_{i-1}-1})\]
is a free linear segment of level $\lambda_i$, where $1 \leq \lambda_i \leq \lambda_{\epsilon}$, and
\[\widehat{\epsilon}_1 = (\epsilon_{s_1}, \epsilon_{s_1+1}, \dots, \epsilon_{s_0})\]
is a free linear segment of level $\lambda_1 = \lambda_{\epsilon}$.

By the paragraph before Lemma \ref{lemma:breakupfreelinear}, we can define $s_1 :=s'$ and we have a free linear segment $\epsilon_{s_1}, \dots, \epsilon_{s_0}$ of level $\lambda_{\epsilon}$. If $s_1 = 1$, then we are done. Otherwise, let $s_1 \leq s_1' < s_0$ be the smallest integer such that $\sum_{j = s_1}^{s_1'} \epsilon_j = -\lambda_{\epsilon}$. If $\epsilon_{s_1-1} = -1$, then $\sum_{j=s_1-1}^{s_1'} = -\lambda_{\epsilon}-1$ and this contradicts the maximality of $\lambda_{\epsilon}$. Hence $\epsilon_{s_1-1}=1$ and $s_1 \geq 2$. As $\sum_{j=1}^s \epsilon_i = 0 = \sum_{j=s_1}^s \epsilon_i$, we have $\sum_{j=1}^{s_1-1} \epsilon_j = 0$. Thus we can speak about the largest integer $s_2 \in \{1, 2, \dots, s_1-2\}$ such that we have $\sum_{j=s_2}^{s_1-1} \epsilon_j= 0$. Hence $\epsilon_{s_2}, \dots, \epsilon_{s_1-1}$ is a free linear segment of level $\lambda_2 \leq \lambda_{\epsilon}$ due to the maximality of $\lambda_{\epsilon}$. Let $s_2 \leq s_2' < s_1$ be the smallest integer such that $\sum_{j=s_2}^{s_2'} \epsilon_j = -\lambda_2$. If $s_2 = 1$, then we are done. Suppose now $s_2 \geq 2$. If $\epsilon_{s_2-1} = -1$, then $\sum_{i=s_2-1}^{s_2'} = -\lambda_2-1$, a contradiction. Thus $\epsilon_{s_2-1} = 1$. Continuing in this fashion, we can find \[s_t:=1 \leq s_{t-1} < s_{t-2} < \cdots < s_1 < s_0:= s\] such that for all $2 \leq i \leq t$, 
\[\epsilon_{s_i}, \epsilon_{s_i+1}, \dots, \epsilon_{s_{i-1}-1}\]
 is a free linear segment of level $\lambda_i$, where $1 \leq \lambda_i \leq \lambda_{\epsilon}$.
\end{proof}

\begin{definition} \label{definition:circularsequencelevel}
Let $\epsilon = (\epsilon_1, \epsilon_2, \dots, \epsilon_s)$ be a circular sequence of integers such that $\epsilon_i \in \{1,0, -1\}$ for all $1 \leq i \leq s$. We call $\epsilon$ a circular sequence of level $\lambda_{\epsilon} \geq 0$ if it satisfies the following conditions:
\begin{enumerate}
\item we have $\sum_{j=1}^s \epsilon_j = 0$;
\item for all $1 \leq a \leq s$ and $a \leq b \leq a+s-1$, we have $|\sum_{j=a}^b \epsilon_j| \leq \lambda_{\epsilon}$;
\item there exist $1 \leq a_0 \leq s$ and $a_0 \leq b_0 \leq a_0+s-1$ such that $|\sum_{j=a_0}^{b_0} \epsilon_j| = \lambda_{\epsilon}$.
\end{enumerate}
\end{definition}

By Proposition \ref{proposition:circularzero}, we know that for every circular sequence $\epsilon = (\epsilon_1, \epsilon_2, \dots, \epsilon_s)$ such that $\epsilon_i \in \{-1,1\}$, if $\Gamma_{\epsilon}^m$ contains a circular digraph, then $\sum_{i=1}^s \epsilon_i = 0$. Given a circular sequence $\epsilon = (\epsilon_1, \epsilon_2, \dots, \epsilon_s)$ of level $\lambda_{\epsilon}$ with $\sum_{i=1}^s \epsilon_i = 0$, we want to know how many circular digraphs there are in $\Gamma_{\epsilon}^m$. By Lemma \ref{lemma:breakupfreelinear}, we know that $\epsilon = (\widehat{\epsilon}_1, \widehat{\epsilon}_2, \dots, \widehat{\epsilon}_t)$, where $\widehat{\epsilon}_i$ is a free linear segment in $\epsilon$ of level $1 \leq \lambda_i \leq \lambda_{\epsilon}$ for all $i \in I_t$, and $\lambda_t = \lambda_{\epsilon}$.\footnote{For notational purpose, we reverse the order of the indices in $\epsilon$ from $\widehat{\epsilon}_1$ to $\widehat{\epsilon}_t$.} For each $i \in I_t$, let $\widehat{\epsilon}_i = (\epsilon_{s_{i,1}}, \epsilon_{s_{i,2}}, \dots, \epsilon_{s_{i,f_i}})$, where $f_i \geq 2$ is the length of $\widehat{\epsilon}_i$. Hence $\sum_{i=1}^t f_i = s$. We know that $\epsilon_{s_{i,1}} = -1$ and $\epsilon_{s_{i,f_i}}=1$. Let 
\[\widetilde{\Gamma}^m_{\widehat{\epsilon}_i} = \bigcup_{1 \leq j \leq f_i-1} \Gamma^m_{\epsilon_{s_i,j}, \epsilon_{s_i,j+1}};\] it is a subgraph of $\Gamma^m_{\epsilon}$. Note that
\[\Gamma^m_{\widehat{\epsilon}_i} = \widetilde{\Gamma}^m_{\widehat{\epsilon}_i} \bigcup \Gamma^m_{\epsilon_{s_i,f_i}, \epsilon_{s_i,1}}.\]
If $\lambda_i \leq m$, then $\widetilde{\Gamma}^m_{\widehat{\epsilon}_i}$ contains a unique free linear digraph of $\Gamma^m_{\epsilon}$ of the form $x_{l_{i,1},s_{i,1}}, x_{l_{i,2},s_{i,2}}, \dots, x_{l_{i,f_i},s_{i,f_i}}$ by Lemma \ref{lemma:pdivbasic3}.

\begin{lemma} \label{lemma:lineardigraphtruncated}
Using the above notation, let $i \in \{1, \dots, t\}$ be such that $\lambda_i \leq m$. Then $\widetilde{\Gamma}^m_{\widehat{\epsilon}_i}$
contains exactly $m-\lambda_i$ linear digraphs of
$\widetilde{\Gamma}^m_{\widehat{\epsilon}_i}$ of length $f_i$ that are not free linear digraph of $\Gamma^m_{\epsilon}$. Moreover, if $\lambda_i < m$, then each of these $m-\lambda_i$ linear digraphs of $\widetilde{\Gamma}^m_{\widehat{\epsilon}_i}$ is of the form:
\begin{equation} \label{equation:lineardigraphtruncated}
x_{l_{i,1}-j,s_{i,1}}, x_{l_{i,2}-j,s_{i,2}}, \dots,
x_{l_{i,f_i}-j,s_{i,f_i}},
\end{equation}
where $j \in \{1, 2, \dots, m-\lambda_i\}$.
\end{lemma}
\begin{proof}
If $\lambda_i \leq m$, then $\widetilde{\Gamma}^m_{\widehat{\epsilon}_i}$ contains a unique free linear digraph of the form $x_{l_{i,1},s_{i,1}}, x_{l_{i,2},s_{i,2}}, \dots, x_{l_{i,f_i},s_{i,f_i}}$ by Lemma \ref{lemma:pdivbasic3}. By the proof of Lemma \ref{lemma:pdivbasic3}, every linear digraph of $\widetilde{\Gamma}^m_{\widehat{\epsilon}_i}$ of length $f_i$ that is not a free linear digraph of $\Gamma^m_{\epsilon}$ must have the form:
\[x_{l_{i,1}-j,s_{i,1}}, x_{l_{i,2}-j,s_{i,2}}, \dots, x_{l_{i,f_i}-j,s_{i,f_i}},\]
for some positive integer $j \geq 1$. As there exists a positive integer $i_0$ such that $l_{i,i_0} = m-\lambda_i$, we must have $1 \leq j \leq m-\lambda_i$ in order for \eqref{equation:lineardigraphtruncated} to be a linear digraph of $\widetilde{\Gamma}^m_{\widehat{\epsilon}_i}$. This completes the proof of the lemma.
\end{proof}

\begin{theorem} \label{theorem:connectedcomponentcal}
Let $\epsilon = (\epsilon_1, \epsilon_2, \dots, \epsilon_s)$ be a circular sequence of level $\lambda_{\epsilon} \geq 0$ such that for each $i \in I_s$, we have $\epsilon_i \in \{1, 0, -1\}$. Then 
\[c(\Gamma^m_{\epsilon}) = \mathrm{max}\{0, m-\lambda_{\epsilon}\}.\]
\end{theorem}
\begin{proof}
If $\lambda_{\epsilon} = 0$, then $\epsilon = (0,\dots,0)$, and it is easy to see that 
\[c(\Gamma^m_{\epsilon}) = m = \mathrm{max}\{0, m-\lambda_{\epsilon}\}\] 
by the construction of $\Gamma_{\epsilon}^m$ (see Case (2) of Subsection \ref{subsection:weighteddigraphs}).

If $\lambda_{\epsilon} > 0$, then we can assume that $\epsilon_i \in \{1,-1\}$ for every $i \in I_s$ by Theorem \ref{secondreduction}. By Lemma \ref{lemma:breakupfreelinear}, we can assume that $\epsilon = (\widehat{\epsilon}_1, \widehat{\epsilon}_2, \dots, \widehat{\epsilon}_t)$, where $\widehat{\epsilon}_i$, $\lambda_i$, $f_i$ and $\Gamma^m_{\widehat{\epsilon}_i}$ are defined as above. For each $i \in I_t$, the intersection of $\Gamma^m_{\widehat{\epsilon}_i}$ and every circular digraph of $\Gamma_{\epsilon}^m$ is a linear digraph of $\Gamma^m_{\widehat{\epsilon}_i}$ (of length $f_i$) but not a free linear digraph of $\Gamma^m_{\epsilon}$ (i.e., as in the form \eqref{equation:lineardigraphtruncated}).

If $\lambda_{\epsilon} > m$ and thus $\lambda_{i_0} > m$ for some $i_0 \in I_t$, then $\Gamma^m_{\widehat{\epsilon}_{i_0}}$ does not contain any linear digraph of $\Gamma^m_{\widehat{\epsilon}_{i_0}}$ of length $f_{i_0}$ by Lemma \ref{lemma:freelineartoolow}. As a result, we have $c(\Gamma^m_{\epsilon}) = 0 = \mathrm{max}\{0, m-\lambda_{\epsilon}\}.$

If $\lambda_{\epsilon} = m$ and thus $\lambda_{i_0} = m$ for some $i_0 \in I_t$, then $\Gamma^m_{\widehat{\epsilon}_{i_0}}$ contains exactly one linear digraph of $\Gamma^m_{\widehat{\epsilon}_{i_0}}$ of length $f_{i_0}$, which is also a free linear digraph of $\Gamma^m_{\epsilon}$. As a result, we have $c(\Gamma^m_{\epsilon}) = 0 = \mathrm{max}\{0, m-\lambda_{\epsilon}\}.$

If $0<\lambda_{\epsilon} < m$, then for each $i \in I_t$, there are $m-\lambda_i$ linear digraphs of $\widetilde{\Gamma}^m_{\widehat{\epsilon}_i}$ of the form \eqref{equation:lineardigraphtruncated} by Lemma \ref{lemma:lineardigraphtruncated}. Hence \[c(\Gamma^m_{\epsilon}) \leq \textrm{min}\{ m -\lambda_i \mid i=1, 2, \dots, t\} = m-\lambda_{\epsilon}.\]

Given a linear digraph of the form \eqref{equation:lineardigraphtruncated}, as $\epsilon_{s_{i-1,f_{i-1}}}=1$ and $\epsilon_{s_{i,1}}=-1$, we know that there exists an edge of weight $0$ in $\Gamma^m_{\epsilon}$ whose source is $ x_{m-1-j,s_{i-1,f_{i-1}}}$ and target is $x_{m-1-j,s_{i,1}}$ for all $1 \leq j \leq m-2$ (see Case (1) of Subsection \ref{subsection:weighteddigraphs}). As $\epsilon_{s_{i,f_i}} = 1$ and $\epsilon_{s_{i+1,1}}=-1$, we know that there exists an edge of weight $0$ in  $\Gamma^m_{\epsilon}$ whose source is $x_{m-1-j,s_{i,f_i}}$ and target is $x_{m-1-j,s_{i+1,f_1}}$ for all $1 \leq j \leq m-2$ (see Case (1) of Subsection \ref{subsection:weighteddigraphs}). Thus for each $1 \leq j \leq m - \lambda_{\epsilon}$, we have a circular digraph of $\Gamma^m_{\epsilon}$:
\[x_{l_{1,1}-j,s_{1,1}}, x_{l_{1,2}-j,s_{1,2}}, \dots, x_{l_{1,f_1}-j,s_{1,f_1}} \dots, x_{l_{t,1}-j,s_{t,1}}, x_{l_{t,2}-j,s_{t,2}}, \dots, x_{l_{t,f_t}-j,s_{t,f_t}}.\]
Hence we also have
\[c(\Gamma^m_{\epsilon}) \geq m-\lambda_{\epsilon} = \mathrm{max}\{0, m-\lambda_{\epsilon}\}.\]
We conlude that in all cases we have
\[c(\Gamma^m_{\epsilon}) = \mathrm{max}\{0, m-\lambda_{\epsilon}\}. \qedhere\]
\end{proof}

\subsection{Main Result}
Using the results in Section \ref{tworeductions} and the previous two subsections of this section, we are ready to state and prove the main result of this paper. We first introduce some notation.

Let $\mathcal{M}_{\pi}$ be an $F$-cyclic $F$-crystal of rank $r$. Recall that $\mathcal{B}_{\pi \times \pi}$ is the set of orbits of $\pi \times \pi$ on $I^2_r$. For every $\mathcal{O} \in \mathcal{B}_{\pi \times \pi}$, 
\begin{enumerate}
	\item if $\epsilon_{\mathcal{O}} = (0, \dots, 0)$ is a circular sequence whose entries are all $0$, then let $\widetilde{\epsilon}_{\mathcal{O}} = \epsilon_{\mathcal{O}}$;
	\item if $\epsilon_{\mathcal{O}} = (\epsilon_1, \epsilon_2, \dots, \epsilon_s)$ is a circular sequence whose entries are not all $0$, then let $\widetilde{\epsilon}_{\mathcal{O}} = (\tilde{\epsilon}_1, \tilde{\epsilon}_2, \dots, \tilde{\epsilon}_t)$ be the circular sequence of integers obtained from $\epsilon_{\mathcal{O}}$ by first removing all zeroes, and then replacing each integer $\epsilon_i \neq 0$ in $\epsilon_{\mathcal{O}}$ with $|\epsilon_i|$ copies of $\epsilon_i/|\epsilon_i|$ (we have $\tilde{\epsilon}_i \in \{-1,1\}$ for all $i \in \{1, 2, \dots, t\}$).
\end{enumerate}
For every nonnegative integer $\lambda$, let $\mathcal{C}_{\pi}(\lambda) \subset \mathcal{B}_{\pi \times \pi}$ be the set orbits $\mathcal{O}$ such that $\widetilde{\epsilon}_{\mathcal{O}}$ is a circular sequence of level $\lambda$. We recall that for $\lambda \geq 1$, $a_{\lambda}(\epsilon_{\mathcal{O}})$ is the number of free linear segments of level $\lambda$ in $\epsilon_{\mathcal{O}}$.

\begin{theorem} \label{theorem:main}
Let $\mathcal{M}_{\pi}$ be an $F$-cyclic $F$-crystal over $k$. Using the above notation, for every integer $m \geq 1$, the dimension of $\mathbf{Aut}_m(\mathcal{M}_{\pi})$ is equal to 
\[\sum_{\mathcal{O} \in \mathcal{B}_{\pi \times \pi}} \sum_{\lambda=1}^m a_{\lambda}(\widetilde{\epsilon}_{\mathcal{O}}),\]
and the number of connected components of $\mathbf{End}_m(\mathcal{M}_{\pi})$ is equal to $p^b$, where
\[b = \sum_{\lambda = 0}^{m-1} \sum_{\mathcal{O} \in \mathcal{C}_{\pi}(\lambda)} (m-\lambda)|\mathcal{O}|.\qedhere\]
\end{theorem}
\begin{proof}
Let $\mathbf{E} := \mathbf{End}_m(\mathcal{M}_{\pi})$. We have 
\begin{align*}
\gamma_{\mathcal{M}_{\pi}}(m) & = \ell(\Gamma_{\mathbf{E}}) & \textrm{by Proposition \ref{proposition:firstcalc}} \\
&= \sum_{\mathcal{O} \in \mathcal{B}_{\pi \times \pi}} \ell(\Gamma^m_{\epsilon_{\mathcal{O}}}) & \textrm{by } \Gamma_{\mathbf{E}} = \bigsqcup_{O \in \mathcal{B}_{\pi \times \pi}} \Gamma^m_{\epsilon_{\mathcal{O}}} \\
&= \sum_{\mathcal{O} \in \mathcal{B}_{\pi \times \pi}} \ell(\Gamma^m_{\widetilde{\epsilon}_{\mathcal{O}}}) & \textrm{by Theorems \ref{firstreduction} and \ref{secondreduction}} \\
&= \sum_{\mathcal{O} \in \mathcal{B}_{\pi \times \pi}} \sum_{\lambda=1}^m a_{\lambda}(\widetilde{\epsilon}_{\mathcal{O}}) & \textrm{by Theorem \ref{theorem:freelinearcal},} \\
\end{align*}
and
\begin{align*}
b & = w(\Gamma_{\mathbf{E}}) & \textrm{by Proposition \ref{proposition:firstcalc}} \\
&= \sum_{\mathcal{O} \in \mathcal{B}_{\pi \times \pi}} w(\Gamma^m_{\epsilon_{\mathcal{O}}}) & \textrm{by } \Gamma_{\mathbf{E}} = \bigsqcup_{O \in \mathcal{B}_{\pi \times \pi}} \Gamma^m_{\epsilon_{\mathcal{O}}} \\
&= \sum_{\mathcal{O} \in \mathcal{B}_{\pi \times \pi}} c(\Gamma^m_{\epsilon_{\mathcal{O}}}) \times |\mathcal{O}| & \textrm{by Corollary \ref{corollary:connectedcomponentfirstcalc}} \\
&= \sum_{\mathcal{O} \in \mathcal{B}_{\pi \times \pi}} c(\Gamma^m_{\widetilde{\epsilon}_{\mathcal{O}}}) \times |\mathcal{O}| & \textrm{by Theorems \ref{firstreduction} and \ref{secondreduction}} \\
&= \sum_{\lambda=0}^{m-1} \sum_{\mathcal{O} \in \mathcal{C}_{\pi}(\lambda)} (m-\lambda)|\mathcal{O}| & \textrm{by Theorem \ref{theorem:connectedcomponentcal}.} & \qedhere
\end{align*}
\end{proof}

\section{Strict Monotonicity of $\mathcal{S}_{\mathcal{M}}^*(1)$}

We begin this section by proving that:

\begin{theorem} \label{theorem:strictinequalitycircular}
Let $\mathcal{M}$ be a nonordinary $F$-circular Dieudonn\'e module of rank $r \geq 2$ (thus $n_{\mathcal{M}} \geq 1$). Then the sequence $\mathcal{S}^*_{\mathcal{M}}(1)$ is strictly decreasing, i.e., we have
\[\gamma_{\mathcal{M}}(1) - \gamma_{\mathcal{M}}(0) > \gamma_{\mathcal{M}}(2) - \gamma_{\mathcal{M}}(1) > \cdots > \gamma_{\mathcal{M}}(n_{\mathcal{M}}) - \gamma_{\mathcal{M}}(n_{\mathcal{M}}-1)>0.\]
\end{theorem}
\begin{proof}
As $\mathcal{M}$ is $F$-circular, there exists an $r$-cycle permutation $\pi$ on $I_r$ such that $\mathcal{M} \cong \mathcal{M}_{\pi}$. As $\mathcal{M}$ is a Dieudonn\'e module, for each $\mathcal{O} \in \mathcal{B}_{\pi \times \pi}$ we have $a_{\lambda}(\widetilde{\epsilon}_{\mathcal{O}}) = a_{\lambda}(\epsilon_{\mathcal{O}})$ by Remark \ref{remark:0dontmatterina}. 

By Theorem \ref{theorem:main}, for every $0 < n <n_{\mathcal{M}}$ we have
\[\gamma_{\mathcal{M}}(n+1) - \gamma_{\mathcal{M}}(n) = \sum_{\mathcal{O} \in \mathcal{B}_{\pi \times \pi}} \sum_{\lambda=1}^{n+1} a_{\lambda}({\epsilon}_{\mathcal{O}}) - \sum_{\mathcal{O} \in \mathcal{B}_{\pi \times \pi}} \sum_{\lambda=1}^{n} a_{\lambda}(\epsilon_{\mathcal{O}}) = \sum_{\mathcal{O} \in \mathcal{B}_{\pi \times \pi}} a_{n+1}({\epsilon}_{\mathcal{O}}).\]
Similarly, we also have
\[\gamma_{\mathcal{M}}(1) - \gamma_{\mathcal{M}}(0) = \sum_{\mathcal{O} \in \mathcal{B}_{\pi \times \pi}} a_{1}({\epsilon}_{\mathcal{O}}) - 0 = \sum_{\mathcal{O} \in \mathcal{B}_{\pi \times \pi}} a_{1}({\epsilon}_{\mathcal{O}}).\]
Thus it suffices to prove that for all $0 \leq n < n_{\mathcal{M}}-1$ we have
\[\sum_{\mathcal{O} \in \mathcal{B}_{\pi \times \pi}} a_{n+2}({\epsilon}_{\mathcal{O}}) < \sum_{\mathcal{O} \in \mathcal{B}_{\pi \times \pi}} a_{n+1}({\epsilon}_{\mathcal{O}}).\]
By Corollary \ref{corollary:freelinearsegmentcontains}, we know that $a_{n+2}({\epsilon}_{\mathcal{O}}) \leq a_{n+1}({\epsilon}_{\mathcal{O}})$. Therefore it is enough to show that there exists $\mathcal{O} \in \mathcal{B}_{\pi \times \pi}$ such that $a_{n+2}({\epsilon}_{\mathcal{O}}) < a_{n+1}({\epsilon}_{\mathcal{O}})$.

Let $(e_1, e_2, \dots, e_{r})$ be the sequence of Hodge slopes of $\mathcal{M}$. We have $e_1, e_2, \dots, e_r \in \{0,1\}$. Let $\mathcal{O}_1, \mathcal{O}_2, \dots, \mathcal{O}_r$ be the orbits of $\pi \times \pi$ on $I_r^2$.  We know that the corresponding $\epsilon_{\mathcal{O}_i}$ sequences are
\begin{align*}
\epsilon_{\mathcal{O}_1} &= (e_1-e_1, e_2-e_2, \dots, e_r-e_r)=(0,0, \dots, 0),\\
\epsilon_{\mathcal{O}_2} &= (e_1-e_2, e_2-e_3, \dots, e_r-e_1),\\
\epsilon_{\mathcal{O}_3} &= (e_1-e_3, e_2-e_4, \dots, e_r-e_2),\\
\vdots \\
\epsilon_{\mathcal{O}_r} &= (e_1-e_r, e_2-e_1, \dots, e_r-e_{r-1}).
\end{align*}
We claim that if $\epsilon_{\mathcal{O}_l}$ does not contain a free linear segment of level $n+1$, then $\epsilon_{\mathcal{O}_{l+1}}$ does not contains a free linear segment of level $n+2$. Indeed, if $\epsilon_{\mathcal{O}_l}$ does not contain a free linear segment of level $n+1$, then it means that for all positive integers $1 \leq u \leq r$ and $1 \leq w \leq r$ we have
\[ \sum_{t=0}^w(e_{u+t}-e_{u+l-1+t}) > -(n+1). \]
Here the subscripts of $e_i$'s are taken modulo $r$. Hence
\begin{gather*}
\sum_{t=0}^w (e_{u+t} - e_{u+l+t}) = \sum_{t=0}^w (e_{u+t} - e_{u+l-1+t}) - e_{u+l+w} + e_{u+l-1} \\ > -(n+1)-e_{u+l+w} \geq -(n+2).
\end{gather*}
This shows that $\epsilon_{\mathcal{O}_{l+1}}$ does not contain a free linear segment of level $n+2$. Note that this claim is true even when $l=r$.

For all $0 \leq n < n_{\mathcal{M}}-1$, as $\gamma_{\mathcal{M}}(n+1) > \gamma_{\mathcal{M}}(n)$, there exists $\mathcal{O} \in \mathcal{B}_{\pi \times \pi}$ such that $a_{n+1}({\epsilon}_{\mathcal{O}}) > 0$. Let $1 < l_0 \leq r$ be the smallest number such that $\epsilon_{\mathcal{O}_{l_0}}$ contains a free linear segment of level $n+1$. As $\mathcal{O} = \mathcal{O}_s$ for some $1 \leq s \leq r$ and $a_{n+1}(\epsilon_{\mathcal{O}}) > 0$, we know that such an $l_0$ exists. If $\epsilon_{\mathcal{O}_{l_0}}$ also contains a free linear segment of level $n+2$, then $\epsilon_{\mathcal{O}_{l_0-1}}$ contains a free linear segment of level $n+1$ by the claim above and this contradicts the minimal property of $l_0$. Hence $\epsilon_{\mathcal{O}_{l_0}}$ does not contain a free linear segment of level $n+2$. We have found a free linear segment of level $n+1$ that is not in any free linear segment of level $n+2$ in $\epsilon_{\mathcal{O}_{l_0}}$. Therefore, $a_{n+2}(\epsilon_{\mathcal{O}_{l_0}}) < a_{n+1}(\epsilon_{\mathcal{O}_{l_0}})$. This completes the proof of the theorem.
\end{proof}

\begin{proposition} \label{proposition:smallimplieslarge}
Let $\mathcal{M}_{\pi}$ be an $F$-cyclic $F$-crystal. Let $\pi = \prod_{i=1}^t \pi_i$ be a decomposition of $\pi$ into disjoint permutations $\pi_i$ and let $\mathcal{M}_{\pi} = \bigoplus_{i=1}^t \mathcal{M}_{\pi_i}$ be the direct sum decomposition of $\mathcal{M}_{\pi}$ into $F$-cyclic $F$-crystals that correspond to $\pi = \prod_{i=1}^t \pi_i$. Let $n \geq 1$ be an integer. If there exists $i_0 \in \{1, 2, \dots, t\}$ such that we have  
\[\gamma_{\mathcal{M}_{\pi_{i_0}}}(n+1)-\gamma_{\mathcal{M}_{\pi_{i_0}}}(n) < \gamma_{\mathcal{M}_{\pi_{i_0}}}(n) - \gamma_{\mathcal{M}_{\pi_{i_0}}}(n-1),\] then we also have \[\gamma_{\mathcal{M}_{\pi}}(n+1)-\gamma_{\mathcal{M}_{\pi}}(n) < \gamma_{\mathcal{M}_{\pi}}(n) - \gamma_{\mathcal{M}_{\pi}}(n-1).\]
\end{proposition}
\begin{proof}
We consider the disjoint union $I_r = \sqcup_{i=1}^t J_i$ such that each $\pi_i$ is a permutation on $J_i$ for $i \in I_t$. Let $\pi_{i_0}$ be the permutation on $J_{i_0} \subset I_r$ and let $\mathcal{B}_{\pi _{i_0}\times \pi_{i_0}}$ be the set of orbits of $\pi_{i_0} \times \pi_{i_0}$ on $J_{i_0}^2$. As
\begin{equation*}
\begin{aligned}
\sum_{\mathcal{O} \in \mathcal{B}_{\pi _{i_0}\times \pi_{i_0}}} a_{n+1}(\widetilde{\epsilon}_{\mathcal{O}}) = & \gamma_{\mathcal{M}_{i_0}}(n+1)-\gamma_{\mathcal{M}_{i_0}}(n) \\ 
&< \gamma_{\mathcal{M}_{i_0}}(n) - \gamma_{\mathcal{M}_{i_0}}(n-1) = \sum_{\mathcal{O} \in \mathcal{B}_{\pi _{i_0}\times \pi_{i_0}}} a_{n}(\widetilde{\epsilon}_{\mathcal{O}}),
\end{aligned}
\end{equation*}
there exists $\mathcal{O} \in \mathcal{B}_{\pi _{i_0} \times \pi_{i_0}}$ such that $a_{n+1}(\widetilde{\epsilon}_{\mathcal{O}}) < a_{n}(\widetilde{\epsilon}_{\mathcal{O}})$. As $\mathcal{B}_{\pi _{i_0} \times \pi_{i_0}} \subset \mathcal{B}_{\pi \times \pi}$, we know that there exists $\mathcal{O} \in \mathcal{B}_{\pi \times \pi}$ such that $a_{n+1}(\widetilde{\epsilon}_{\mathcal{O}}) < a_{n}(\widetilde{\epsilon}_{\mathcal{O}})$. From this and Theorem \ref{theorem:main}, we get that 
\[\gamma_{\mathcal{M}_{\pi}}(n+1)-\gamma_{\mathcal{M}_{\pi}}(n) < \gamma_{\mathcal{M}_{\pi}}(n) - \gamma_{\mathcal{M}_{\pi}}(n-1). \qedhere\]
\end{proof}

Before we prove the next corollary, we recall the notion minimal Dieudonn\'e module after \cite[1.5.1 Definition]{Vasiu:reconstructing}. Let $\mathcal{M}$ be an isoclinic Dieudonn\'e module, i.e., its Newton polygon is a straight line. Let $\mathcal{E} = (e_1, e_2, \dots, e_r)$ be the sequence of Hodge slopes of $\mathcal{M}$. Let $\lambda_{\mathcal{M}} = \sum_{i=1}^r e_i/r$ be the unique Newton slope. We say that $\mathcal{M} = \mathcal{M}_{\pi}$ is isoclinic minimal if there exists an ordered $W(k)$-basis $(v_1, v_2, \dots, v_r)$ of $M$ such that for all for all $1 \leq i, q \leq r$ we have
\[\varphi_{\pi}^q(v_i) = p^{\lfloor q\lambda_{\mathcal{M}} \rfloor + \epsilon_q(i)}v_{\pi^q(i)},\]
where $\epsilon_q(i) \in \{0,1\}$. In general, we say that $\mathcal{M}$ is minimal if it is a direct sum of isoclinic minimal Dieudonn\'e modules. A Dieudonn\'e module $\mathcal{M}$ is minimal if and only if $n_{\mathcal{M}} \leq 1$ (\cite[1.6 Main Theorem B]{Vasiu:reconstructing}).

\begin{corollary} \label{corollary:strictinequalityfcyclic}
If $\mathcal{M}$ is a nonordinary $F$-cyclic Dieudonn\'e module, then we have $\gamma_{\mathcal{M}}(2) - \gamma_{\mathcal{M}}(1) < \gamma_{\mathcal{M}}(1) - \gamma_{\mathcal{M}}(0) = \gamma_{\mathcal{M}}(1)$.
\end{corollary}
\begin{proof}
If $\mathcal{M}$ is minimal and nonordinary, then $n_{\mathcal{M}} = 1$ (\cite[1.6 Main Theorem B]{Vasiu:reconstructing}) and thus $\gamma_{\mathcal{M}}(2) - \gamma_{\mathcal{M}}(1) = 0 < \gamma_{\mathcal{M}}(1)$.

If $\mathcal{M}$ is not minimal, then by Theorem \ref{theorem:strictinequalitycircular}, we can assume that $\mathcal{M}$ is not $F$-circular. Let $\mathcal{M} = \mathcal{M}_{\pi}$ and $\pi = \prod_{i=1}^t \pi_i$ be a decomposition of $\pi$ into at least two disjoint cycles and let $\mathcal{M}_{\pi_i}$ be the direct summand of $\mathcal{M}$ defined by $\pi_i$. There exists at least one $i_0 \in \{1, 2, \dots, t\}$ such that $\mathcal{M}_{\pi_{i_0}}$ is not minimal by \cite[Theorem 1.2]{Xiao:minimal}.

As $\mathcal{M}_{\pi_{i_0}}$ is not minimal, we know that $n_{\mathcal{M}} \geq 2$ (\cite[1.6 Main Theorem B]{Vasiu:reconstructing}) and thus $\gamma_{\mathcal{M}_{\pi_{i_0}}}(2) > \gamma_{\mathcal{M}_{\pi_{i_0}}}(1)$ by Theorem \ref{theorem:vasiugabberbasic}. From Theorem \ref{theorem:strictinequalitycircular}, we get that $\gamma_{\mathcal{M}_{\pi_{i_0}}}(2) - \gamma_{\mathcal{M}_{\pi_{i_0}}}(1) < \gamma_{\mathcal{M}_{\pi_{i_0}}}(1)$. Hence $\gamma_{\mathcal{M}_{\pi}}(2) - \gamma_{\mathcal{M}_{\pi}}(1) < \gamma_{\mathcal{M}_{\pi}}(1)$ by Proposition \ref{proposition:smallimplieslarge}.
\end{proof}

\begin{corollary}
Let $\mathcal{M}$ be a nonordinary $F$-cyclic Dieudonn\'e module. For every $i > j \geq 1$, we have $\gamma_{\mathcal{M}}(i)/\gamma_{\mathcal{M}}(j) < i/j.$
\end{corollary}
\begin{proof}
Let $j=1$ and we prove that $\gamma_{\mathcal{M}}(i)/\gamma_{\mathcal{M}}(1) < i$ by induction on $i \geq 2$. The base step when $i=2$ follows from Corollary \ref{corollary:strictinequalityfcyclic}. Suppose that $\gamma_{\mathcal{M}}(l)/\gamma_{\mathcal{M}}(1) < l$ for some $l \geq 2$. As $\gamma_{\mathcal{M}}(l+1) - \gamma_{\mathcal{M}}(1) \leq \gamma_{\mathcal{M}}(l)$ by \cite[Proposition 2.11]{Xiao:vasiuconjecture}, we have $\gamma_{\mathcal{M}}(l+1) \leq \gamma_{\mathcal{M}}(l) + \gamma_{\mathcal{M}}(1) < (l+1) \gamma_{\mathcal{M}}(1)$ by induction.

We now prove $\gamma_{\mathcal{M}}(i)/\gamma_{\mathcal{M}}(j) < i/j$ by induction on $j \geq 1$. The base step when $j=1$ follows from the last paragraph. Suppose that $\gamma_{\mathcal{M}}(i)/\gamma_{\mathcal{M}}(l) < i/l$ for some $l \geq 1$ and for all $i > l$. We want to show that $\gamma_{\mathcal{M}}(i)/\gamma_{\mathcal{M}}(l+1) < i/(l+1)$ for all $i > l+1$.

Base step: Suppose $i=l+2$. Because $\gamma_{\mathcal{M}}(l+2)-\gamma_{\mathcal{M}}(l+1) \leq \gamma_{\mathcal{M}}(l+1) - \gamma_{\mathcal{M}}(l)$, we have, $\gamma_{\mathcal{M}}(l+2)/\gamma_{\mathcal{M}}(l+1) \leq 2 - \gamma_{\mathcal{M}}(l)/\gamma_{\mathcal{M}}(l+1)$. By the inductive hypothesis, we have $\gamma_{\mathcal{M}}(l+1)/\gamma_{\mathcal{M}}(l) < (l+1)/l$, therefore $\gamma_{\mathcal{M}}(l+2)/\gamma_{\mathcal{M}}(l+1) < 2-l/(l+1)= (l+2)/(l+1).$

Inductive step: Suppose that $\gamma_{\mathcal{M}}(i)/\gamma_{\mathcal{M}}(l+1) < i/(l+1)$ for some $i > l+1$. As $\gamma_{\mathcal{M}}(i+1)-\gamma_{\mathcal{M}}(l+1) \leq \gamma_{\mathcal{M}}(i) - \gamma_{\mathcal{M}}(l)$, we have
\[\frac{\gamma_{\mathcal{M}}(i+1)}{\gamma_{\mathcal{M}}(l+1)} \leq 1 + \frac{\gamma_{\mathcal{M}}(i)}{\gamma_{\mathcal{M}}(l+1)} - \frac{\gamma_{\mathcal{M}}(l)}{\gamma_{\mathcal{M}}(l+1)} < 1 + \frac{i}{l+1} - \frac{l}{l+1} = \frac{i+1}{l+1}.\]
This completes the proof of the corollary.
\end{proof}

\begin{example} \label{example:counterexamplefcrystal}
Let $\mathcal{M}_{\pi} = (M, \varphi_{\pi})$ be an $F$-circular $F$-crystal of rank $2$, where $M$ has an ordered $W(k)$-basis $(v_1, v_2)$, $\pi = (1,2)$ and $(e_1,e_2) = (0,e)$, where $e \geq 2$ is an integer. Hence $\varphi(v_1)=v_2, \varphi(v_2) = p^ev_1$. There are two orbits of $\pi \times \pi$ on $I_2^2$ with
\[\epsilon_{\mathcal{O}_1} = (0,0), \qquad \epsilon_{\mathcal{O}_2} = (-e,e).\]
It is easy to compute $a_{\lambda}(\epsilon_{\mathcal{O}_1}) = 0$ for all $\lambda \geq 1$, $a_{\lambda}(\epsilon_{\mathcal{O}_2}) = 1$ for all $1 \leq \lambda \leq e$. Therefore $\gamma_{\mathcal{M}}(0) = 0$, $\gamma_{\mathcal{M}}(i) = i$ for all $1 \leq i \leq e$, and $\gamma_{\mathcal{M}}(i) = e$ for all $i \geq e$. Hence $\gamma_{\mathcal{M}}(n+1) - \gamma_{\mathcal{M}}(n) = 1$ for all $0 \leq n < e$. Moreover $n_{\mathcal{M}} = e$ (\cite[1.5.2 Theorem]{Vasiu:reconstructing}). Thus
\[\gamma_{\mathcal{M}}(2) - \gamma_{\mathcal{M}}(1) = \gamma_{\mathcal{M}}(3) - \gamma_{\mathcal{M}}(2) = \cdots = \gamma_{\mathcal{M}}(n_{\mathcal{M}}) - \gamma_{\mathcal{M}}(n_{\mathcal{M}}-1) > 0,\]
which shows that in Theorem \ref{theorem:strictinequalitycircular} the assumption that $\mathcal{M}$ is a Dieudonn\'e module is necessary.
\end{example}

\section*{Statement of Name Change}
The case where $\mathcal{M}_{\pi}$ is a Dieudonn\'e module in the main result was done in the dissertation \cite{Ding:thesis} of the first author. Since then, he has changed his name from Ding Ding to Zeyu Ding.

\section*{Acknowledgment}
The authors would like to thank Adrian Vasiu for many suggestions during the preparation of this paper and to the first draft of this paper. The first author would like to thank Adrian Vasiu for introducing him to $p$-divisible groups and for his guidance, encouragement and support throughout his doctoral studies at Binghamton University.

\bibliographystyle{amsplain}
\bibliography{references}

\providecommand{\bysame}{\leavevmode\hbox to3em{\hrulefill}\thinspace}
\providecommand{\MR}{\relax\ifhmode\unskip\space\fi MR }
\providecommand{\MRhref}[2]{%
  \href{http://www.ams.org/mathscinet-getitem?mr=#1}{#2}
}
\providecommand{\href}[2]{#2}
\begin{thebibliography}{1}

\bibitem{Ding:thesis}
Ding Ding, \emph{Canonical {B}arsotti-{T}ate groups of finite level}, Ph.D.
  thesis, Binghamton University, State University of New York, Binghamton, NY,
  December 2015.

\bibitem{Vasiu:dimensions}
Ofer Gabber and Adrian Vasiu, \emph{Dimensions of group schemes of
  automorphisms of truncated {B}arsotti--{T}ate groups}, Int. Math. Res. Not.
  \textbf{18} (2013), 4285--4333.

\bibitem{Kraft1}
Hanspeter Kraft, \emph{{Kommutative algebraische $p$-Gruppen (mit Anwendungen
  auf $p$-divisible Gruppen und abelsche Variet\"aten)}}, Univ. Bonn, 86pp,
  September 1975.

\bibitem{Vasiu:CBP}
Adrian Vasiu, \emph{Crystalline boundedness principle}, Ann. Sci. \'{E}c. Norm.
  Sup. (4) \textbf{39} (2006), no.~2, 245--300.

\bibitem{Vasiu:modp}
\bysame, \emph{Mod $p$ classification of {S}himura {$F$}-crystals}, Math.
  Nachr. \textbf{283} (2010), no.~8, 1068--1113.

\bibitem{Vasiu:reconstructing}
\bysame, \emph{Reconstructing $p$-divisible groups from their truncations of
  small level}, Comment. Math. Helv. \textbf{85} (2010), no.~1, 165--202.

\bibitem{Xiao:vasiuconjecture}
Xiao Xiao, \emph{Subtle invariants of {$F$}-crystals}, J. Ramanujan Math. Soc.
  \textbf{29} (2014), no.~4, 413--458.

\bibitem{Xiao:minimal}
\bysame, \emph{Minimal {$F$}-crystals and isomorphism numbers of isosimple
  {$F$}-crystals}, Math. Nachr. \textbf{290} (2017), 1406--1419.

\end{thebibliography}

\end{document}